\documentclass{article}[twosided]

\newif\ifaistats
\newif\ifarxiv
\newif\ifneurips
\newif\ificml
\newif\ifextension
\newif\ifaistatsnot
\newif\ifarxivnot
\newif\ifneuripsnot
\newif\ificmlnot
\newif\ifextensionnot

\aistatsfalse
\arxivfalse
\neuripsfalse
\icmlfalse
\extensionfalse
\aistatsnottrue
\arxivnottrue
\neuripsnottrue
\icmlnottrue
\extensionnottrue

\arxivtrue

\ifaistats \aistatsnotfalse \fi
\ifarxiv \arxivnotfalse \fi
\ifneurips \neuripsnotfalse \fi
\ificml \icmlnotfalse \fi
\ifextension \extensionnotfalse \fi

\ifaistats \aistatsnotfalse \fi
\ifarxiv \arxivnotfalse \fi
\ifneurips \neuripsnotfalse \fi
\ificml \icmlnotfalse \fi
\ifextension \extensionnotfalse \fi

\usepackage[utf8]{inputenc} 
\usepackage[T1]{fontenc}    
\usepackage{hyperref}       
\usepackage{url}            
\usepackage{booktabs}       
\usepackage{amsfonts}       
\usepackage{nicefrac}       
\usepackage{microtype}      
\usepackage{xcolor}         
\usepackage{wrapfig}
\ifaistatsnot
    \usepackage{natbib}
\fi
\usepackage{xargs}                      
\usepackage{comment}

\usepackage{hyperref}
\usepackage{url, times}

\usepackage{amssymb,amsmath,amscd,amsfonts,amstext,amsthm,bbm, enumerate, dsfont, mathtools}
\usepackage{array}
\usepackage{multicol}
\usepackage{graphicx}
\usepackage{varwidth}
\usepackage{hyperref}
\usepackage{import}
\usepackage{caption,subcaption}
\ificmlnot
    \usepackage{algorithm}
    \usepackage{algpseudocode}
\fi
\usepackage[none]{hyphenat}
\usepackage{footnote}
\usepackage{cleveref}
\usepackage{makecell}
\usepackage{threeparttable}
\usepackage{booktabs}
\usepackage{tcolorbox}
\usepackage{caption}
\usepackage{subcaption}

\usepackage{scrextend}

\graphicspath{ {images/} }

\newtheorem{theorem}{Theorem}
\newtheorem{proposition}{Proposition}
\newtheorem{lemma}{Lemma}

\newtheorem{assumption}{Assumption}

\newtheorem{definition}{Definition}

\DeclareMathOperator*{\argmin}{argmin}

\newcommand{\EE}{\mathbb{E}}
\newcommand{\ec}[2][]{\ensuremath{\mathbb{E}_{#1} \left[#2\right]}}
\newcommand{\R}{\mathbb{R}}

\newcommand{\eqdef}{\stackrel{\text{def}}{=}}
\newcommand{\cO}{{\cal O}}

\newcommand{\opt}{x^*}

\usepackage{tablefootnote}
\usepackage{changepage}
\usepackage{stackengine}
\usepackage{pifont}

\usepackage{mdframed}
\mdfsetup{%
	linecolor=white,
	backgroundcolor=blue!5,
}
\newcolumntype{?}{!{\vrule width 1pt}}
\usepackage{colortbl}

\definecolor{mydarkgreen}{RGB}{39,130,67}

\definecolor{mydarkred}{RGB}{192,47,25}

\newcommand{\cmark}{{\color{mydarkgreen}\ding{51}}}%
\newcommand{\xmark}{{\color{mydarkred} \ding{55}}}%
\newcommand{\norm}[1]{{\left \| #1 \right\|}}
\newcommand{\E}[1]{\mathbb{E}\left[#1\right] }

\newcommand{\Range}[1]{{\rm Range}\left( #1\right)}
\newcommand{\fopt}{f^*}
\newcommand{\Null}[1]{{\rm Null}\left( #1\right)}
\newcommand{\efa}{&& \hspace{-1.5cm}}
\newcommand{\Rd}{\mathbb R^d}

\newcommand{\Lstandard}{{L_{\text{sc}}}}

\newcommand{\Lsemi}{{L_{\text{semi}}}}
\newcommand{\Lstrongly}{{L_{\text{str}}}}
\newcommand{\Lalg}{{L_{\text{alg}}}}

\newcommand{\cD}{\mathcal D}
\newcommand{\s}{\mathbf S}
\newcommand{\sk}{\mathbf S_k}
\newcommand{\gS}{\nabla_{\s} f}
\newcommand{\gSk}{\nabla_{\sk} f}
\newcommand{\hS}{\nabla_{\s}^2 f}
\newcommand{\hSk}{\nabla_{\sk}^2 f}

\newcommand{\st}{\s ^ \top}
\newcommand{\Ls}{L_{\s}}
\newcommand{\als}[1]{\alpha_{#1}}
\newcommand{\modelS}[1]{T_{\s} \left(#1\right)}
\newcommand{\modelSk}[1]{T_{\sk} \left(#1\right)}

\newcommand{\normMS}[2]{{\left \| #1 \right\|}_{#2, \s}}
\newcommand{\normsMS}[2]{{\left \| #1 \right\|}_{#2, \s}^2}
\newcommand{\normMSd}[2]{{\left \| #1 \right\|}_{#2, \s}^*}
\newcommand{\normsMSd}[2]{{\left \| #1 \right\|}_{#2, \s}^{*2}}
\newcommand{\normMSdk}[2]{{\left \| #1 \right\|}_{#2, \sk}^*}
\newcommand{\normsMSdk}[2]{{\left \| #1 \right\|}_{#2, \sk}^{*2}}

\newcommand{\level}{\mathcal Q(x_0)}
\newcommand{\p}{\mathbf P_x}
\newcommand{\pk}[1]{\mathbf P_{x^{#1}}}
\newcommand{\murel}{\hat \mu}
\newcommand{\Lrel}{\hat L}
\newcommand{\Lrels}{\hat L_{\sk}}

\newcommand{\okd}{\mathcal O \left( k^{-2} \right)}

\newcommand{\algstyle}{{\sf}}
\newcommand{\aicn}{{\algstyle AICN}}
\newcommand{\sgn}{{\algstyle SGN}}
\newcommand{\sscn}{{\algstyle SSCN}}
\newcommand{\rsn}{{\algstyle RSN}}
\newcommand{\cd}{{\algstyle CD}}
\newcommand{\acd}{{\algstyle ACD}}
\newcommand{\libsvm}{{\algstyle LIBSVM}}

\ifextension

\fi

\newcommand{\figsize}{0.24\textwidth}
\newcommand{\figsizeap}{0.49\textwidth}
\newcommand{\tn}[1]{\tnote{\color{red}(#1)}}
\newcommand{\afrac}{\frac}
\newcommand{\Tr}{\text{Tr}}

\newcommand{\h}{\nabla^2 f}
\newcommand{\g}{\nabla f}

\newcommand{\mI}{\mathbf I}

\newcommand{\normM}[2]{{\left \| #1 \right\|}_{#2}}
\newcommand{\normsM}[2]{{\left \| #1 \right\|}_{#2}^2}
\newcommand{\normMd}[2]{{\left \| #1 \right\|}_{#2}^*}
\newcommand{\normsMd}[2]{{\left \| #1 \right\|}_{#2}^{*2}}

\usepackage{theoremref}

\usepackage[colorinlistoftodos,bordercolor=orange,backgroundcolor=orange!20,linecolor=orange,textsize=scriptsize]{todonotes}

\newcommand{\sumin}[2]{ \sum \limits_{#1=1}^{#2}}

\newcommand{\N}{\mathbb N}

\newcommand{\ip}[2]{\left\langle #1, #2  \right \rangle}

\def\<#1,#2>{\left\langle #1,#2\right\rangle}

\usepackage{enumitem}
\newlist{myitemize}{itemize}{3}
\setlist[myitemize,1]{label=\textbullet,leftmargin=0in}

\newcommand{\lb}{\left\lbrace}
\newcommand{\rb}{\right\rbrace}
\newcommand{\la}{\left\langle}
\newcommand{\ra}{\right\rangle}

\newcommand{\bbE}{\mathbb{E}}

\ificml
    \newcommand{\State}{\STATE}
    
    \newcommand{\For}{\FOR}
    \newcommand{\EndFor}{\ENDFOR}

\fi

\newcommand{\sap}{{sketch-and-project}}
\newcommand{\Sap}{{ Sketch-and-project}}
\newcommand{\en}{{ Exact Newton descent}}
\newcommand{\newton}{{Newton}}
\newcommand{\cnewton}{{cubic Newton}}
\newcommand{\Cnewton}{{Cubic Newton}}

\newcommand{\Dnewton}{{Damped Newton}}
\newcommand{\rnewton}{{regularized Newton}}
\newcommand{\Rnewton}{{Regularized Newton}}

\newcommand{\Gnewton}{{Glob.~reg.~Newton}}

\newcommand{\vspacefig}{0mm}

\ifaistats
    \usepackage{\string~"./styles/aistats2024"}
    
    \usepackage[round]{natbib}

\fi

\ifarxiv
    \usepackage{\string~"./styles/arxiv"}
    \author{
    Slavom\'ir Hanzely\footremember{MBZUAI} {Mohammed bin Zayed University of Artificial Intelligence}
    }
    \date{}
    \title{Sketch-and-Project Meets Newton Method: \newline Global $\mathcal O \left( k^{-2} \right)$ Convergence with Low-Rank Updates}
\fi

\ifneurips
    \usepackage{\string~"./styles/neurips_2024"}
    \author{
    Slavom\'ir Hanzely\footremember{KAUST} {King Abdullah University of Science and Technology, Saudi Arabia}
    }
    \date{}
    \title{Sketch-and-Project Meets Newton Method: \newline Global $\mathcal O \left( k^{-2} \right)$ Convergence with Low-Rank Updates}
\fi

\ificml
    \usepackage{\string~"./styles/icml2024"}
    \icmltitlerunning{Newton Method with Global $\mathcal O \left( k^{-2} \right)$ Rate and Low-Rank Updates}
    \author{
    }
    \date{}
    \title{Sketch-and-Project Meets Newton Method: \newline Global $\mathcal O \left( k^{-2} \right)$ Convergence with Low-Rank Updates}

\fi

\begin{document}
    \ifaistats
    
    \twocolumn[
    \aistatstitle{Sketch-and-Project Meets Newton Method: Global $\mathcal O \left( k^{-2} \right)$ Convergence with Low-Rank Updates}
    \aistatsauthor{ Author 1 \And Author 2 \And  Author 3 }
    \aistatsaddress{ Institution 1 \And  Institution 2 \And Institution 3 } ]
    \fi

    \ifaistatsnot
	   \maketitle
    \fi
 
	\begin{abstract} 
    In this paper, we propose the first sketch-and-project Newton method with the fast $\mathcal O \left( k^{-2} \right)$ global convergence rate for self-concordant functions. 
    Our method, \sgn{}, can be viewed in three ways: i) as a sketch-and-project algorithm projecting updates of the Newton method, ii) as a cubically regularized Newton method in the sketched subspaces, and iii) as a damped Newton method in the sketched subspaces.

    \sgn{} inherits the best of all three worlds: the cheap iteration costs of the sketch-and-project methods, the state-of-the-art $\mathcal O \left( k^{-2} \right)$ global convergence rate of the full-rank Newton-like methods, and the algorithm simplicity of the damped Newton methods. 
    Finally, we demonstrate its comparable empirical performance to the baseline algorithms.
	\end{abstract}

\setlist[itemize]{itemsep=-.2em}
\setlist{leftmargin=*}

\section{Introduction}
Second-order methods have always been fundamental for both scientific and industrial computing. Their rich history can be traced back to the works \citet{Newton}, \citet{Raphson}, and \citet{Simpson}, and they have undergone extensive development since \citep{kantorovich1948functional, more1978levenberg, griewank1981modification}. For the more historical development of classical methods, we refer the reader to \citet{ypma1995historical}. The number of practical applications is enormous, with over a thousand papers included in the survey of \citet{conn2000trust} on trust-region and quasi-Newton methods alone.

Second-order methods are highly desirable due to their invariance to rescalings and coordinate transformations, which significantly reduces the complexity of hyperparameter tuning. Moreover, this invariance allows convergence independent of the conditioning of the underlying problem. In contrast, the convergence rate of first-order methods is fundamentally dependent on the function conditioning. Moreover, the first-order methods can be sensitive to variable parametrization and function scale, hence parameter tuning (e.g., step size) is often crucial for efficient execution.

On the other hand, even the simplest and most classical second-order method, the Newton method, achieves an extremely fast, quadratic convergence rate (precision doubles in each iteration) when initialized sufficiently close to the solution. However, the convergence of the Newton method is limited only to the neighborhood of the solution. Several works, including \citet{jarre2016simple}, \citet{mascarenhas2007divergence}, \citet{bolte2022curiosities} demonstrate that when initialized far from optimum, the line search and the trust-region Newton-like methods can diverge on both convex and nonconvex problems.

\begin{table*}[t!]
    \centering
    \setlength\tabcolsep{6pt} 
    \begin{threeparttable}[t]
        {\scriptsize
            \renewcommand\arraystretch{2} 
            \caption[Global convex convergence rates of low-rank Newton methods]{
            Global convergence rates of low-rank Newton methods for convex and smooth functions. We report dependence on the number of iterations $k$. We use the fastest full-dimensional algorithms as the baseline, we highlight the best rate in blue.}
            \label{tab:setup_comparison}
            \centering 
            \begin{tabular}{?c?c|c?}
            \Xhline{2\arrayrulewidth}
             \makecell{\\ \\ \textbf{Update} \\ \textbf{direction}} \qquad \makecell{ \textbf{Update}\\ \textbf{oracle} \\ \\} & \makecell{\textbf{Full-dimensional} \\ (direction is deterministic)} & \makecell{\textbf{Low-rank} \\ (direction in expectation)}\\
             
             \Xhline{2\arrayrulewidth}
             \makecell{\textbf{Non-Newton} \\ \textbf{direction} }
             & \makecell{ {\color{blue} $\cO(k^{-2})$} \\
             Cubically regularized Newton \\  \citep{nesterov2006cubic}, \\
             Globally regularized Newton \\ \citep{mishchenko2021regularized, doikov2021optimization} } 
             
             & \makecell{$\cO(k^{-1})$ \\ 
              Stochastic Subspace Cubic Newton \\ \citep{hanzely2020stochastic} }\\               

             \hline
             \makecell{\textbf{Newton} \\ \textbf{direction}} 
             & \makecell{ {\color{blue} $\cO(k^{-2})$} \\ 
              Affine-Invariant Cubic Newton  \\ \citep{hanzely2022damped}}
             
             & \makecell{ {\color{blue} $\cO(k^{-2})$} \\ 
            \textbf{Sketchy Global Newton}  \textbf{(new)} \\ \hline
             {$ \cO(k^{-1})$}\\ 
              Randomized Subspace Newton \\ \citep{RSN}}\\
             \Xhline{2\arrayrulewidth}
            \end{tabular}
        }
    \end{threeparttable}
    \vspace{\vspacefig}
\end{table*}
\begin{algorithm}[t]
    \caption{\sgn{}: Sketchy Global Newton \textbf{(new)}} \label{alg:sgn}
    \begin{algorithmic}[1]
        \State \textbf{Requires:} Initial point $x^0 \in \R^d$, distribution of sketch matrices $\cD$ such that $\mathbb E_{\s \sim \cD} \left[ \p \right] = \afrac \tau d \mI$, 
		upper bound on semi-strong self-concordance constant $\Lalg \geq \Lsemi$
        \For {$k=0,1,2\dots$}
        \State Sample $\s_k \sim \cD$
        \State $\als k = \frac {-1 +\sqrt{1+2 \Lalg \normMSdk{\gSk(x^k)} {x^k} }}{\Lalg \normMSdk {\gSk(x^k)} {x^k} }$
        \State $x^{k+1} = x^k - \als k \s_k \left[\hSk(x^k)\right]^{\dagger} \gSk(x^k)$        
        \EndFor
    \end{algorithmic}
\end{algorithm}

\subsection{Demands of modern machine learning}
Despite the long history, research on second-order methods has been thriving to this day. 
Newton-like methods with the fast $\okd$ global rate were introduced relatively recently under the names Cubic Newton method \citep{nesterov2006cubic} or Globally regularized Newton methods \citep{doikov2021local,mishchenko2021regularized,hanzely2022damped}. The main limitation of these methods is their poor scalability for modern large-scale machine learning. Large datasets with numerous features necessitate well-scalable algorithms. While tricks and inexact approximations can be used to avoid computing the inverse Hessian, simply storing the Hessian becomes impractical when the dimensionality $d$ is large.
This challenge has served as a catalyst for the recent developments. To address the curse of the dimensionality, \citet{SDNA}, \citet{luo2016efficient}, \citet{RSN}, \citet{RBCN}, and \citet{hanzely2020stochastic} proposed Newton-like methods operating in random low-dimensional subspaces.
This approach, also known as \sap{} \citep{gower2015randomized}, substantially reduces the computational cost per iteration. However, this happens at the cost of slower, $\cO \left( k^{-1}\right)$, convergence rate \citep{gower2020variance, hanzely2020stochastic}.

\subsection{Contributions}
In this work, we argue that the \sap{} adaptations of second-order methods can be improved. To this end, we introduce the \textbf{first} \sap{} method (\sgn{}, \Cref{alg:sgn}) which boasts a global \textbf{$\okd$} convex convergence rate, matching dependence on iteration number $k$ of full-dimensional \rnewton{} methods. 
Surprisingly, sketching on $1$-dimensional subspaces with an iteration cost of $\cO(1)$ \citep{RSN} engenders $\okd$ global convex rate. 

As a cherry on top, \sgn{} offers additional benefits in the form of a local linear convergence rate independent of the condition number and a global linear rate under the assumption of relative convexity (\Cref{def:rel}).
We summarize the contributions below and in Tables \ref{tab:three_ways}, \ref{tab:detailed}.


\begin{itemize} 
\item \textbf{One connects all:} We present \sgn{} through three orthogonal viewpoints: the \sap{} method, the subspace \newton{} method, and the subspace \rnewton{} method. Compared to the established algorithms, \sgn{} can be viewed as \aicn{} \citep{hanzely2022damped} operating in subspaces, \sscn{} \citep{hanzely2020stochastic} operating in local norms, or \rsn{} \citep{RSN} with the new stepsize schedule (\Cref{tab:three_ways}).

Designing an algorithm that preserves all desired properties was a significant challenge. It required a multitude of insights from the literature and an extremely careful analysis technique. Even the smallest misalignment can cause significantly slower convergence rate guarantees (see \Cref{sec:sscn_okd}). 
\item
\textbf{Fast global convergence:}
\sgn{} is the \textbf{first low-rank method} that minimizes convex functions with $\okd$ global rate. This matches the state-of-the-art rates of full-rank Newton-like methods. Other \sap{} methods (e.g., \sscn{} and \rsn{}), have slower $\mathcal O \left( k^{-1} \right)$ rate (\Cref{tab:setup_comparison}).

\item \textbf{Cheap iterations:} \sgn{} uses $\tau$--dimensional updates. Per-iteration cost is $\mathcal O \left( d\tau^2 \right)$ and in the $\tau=1$ case it is even $\mathcal O \left(1\right)$ \citep{RSN}. Conversely, full-rank Newton-like methods have cost proportional to $d^3$  and $d \gg \tau$.

\item \textbf{Linear local rate:} \sgn{} has local linear rate $\cO \left( \frac d \tau \log \frac 1 \varepsilon \right)$ (\Cref{th:local_linear}) dependent only on the ranks of the sketching matrices. This improves over the condition-dependent linear rate of \rsn{} or any rate of first-order methods.

\item \textbf{Global linear rate:}  Under $\murel$--relative convexity, \sgn{} with a different smoothness constant achieves global linear rate $\cO \left( \frac {\Lalg} {\rho \murel } \log \frac 1 \varepsilon \right)$ to a neighborhood of the solution (\Cref{th:global_linear})\footnote{$\rho$ is condition number of a projection matrix \eqref{eq:rho}, and $\Lalg$ is upperbound on semi-strong self-concordance (\Cref{def:semi-self-concordance}) affecting the stepsize \eqref{eq:sgn_alpha}.}. 

\item \textbf{Geometry and interpretability:}
Update of \sgn{} uses well-understood projections\footnote{\citet{gower2020variance} describes six equivalent viewpoints.} of Newton method with stepsize schedule \aicn{}. Moreover, those stochastic projections are affine-invariant and in expectation preserve direction \eqref{as:projection_direction}. 
On the other hand, implicit steps of regularized Newton methods including \sscn{} lack geometric interpretability.

\item \textbf{Algorithm simplicity:}
\sgn{} is affine-invariant and independent of the choice of the basis, simplifying parameter tuning.
Update rule \eqref{eq:sgn_aicn} is simple and explicit. Conversely, most fast globally convergent Newton-like algorithms require an extra subproblem solver in each iteration.

\item \textbf{Analysis:}
The analysis of \sgn{} is simple, all steps have clear geometric interpretation. On the other hand, the analysis of \sscn{} \citep{hanzely2020stochastic} is complicated as it measures distances in both $l_2$ norms and local norms. Using $l_2$ norms removes geometric interpretability and leads to worse constants, which ultimately causes the slower $\mathcal O \left( k^{-1} \right)$ convergence rate.
\end{itemize}

\subsection{Objective}
In this chapter, we consider the optimization objective
\begin{equation} \label{eq:objective}
    \min_{x \in \Rd} f(x),
\end{equation}
where function $f:\R^d \to \R$ is convex, twice differentiable with positive definite Hessian, bounded from below, and potentially ill-conditioned. The number of features $d$ is large. Denote the solution $x^* \eqdef \argmin_{x\in\Rd} f(x)$ and $ \fopt\eqdef f(x^*)$. We solve objective \eqref{eq:objective} using subspace methods, which use a sparse update
\begin{equation} \label{eq:subspace_generic}
x_+ = x + \s h,
\end{equation}
where $\s \in \mathbb{R}^{d \times \tau(\s)}, \s \sim \cD$ is a thin matrix and $h \in \mathbb{R}^{\tau(\s)}$. We denote gradients and Hessians along the subspace spanned by columns of $\s$ as {$\gS (x) \eqdef \st \g (x)$} and {$\hS (x) \eqdef \st \h(x) \s$}.
\citet{RSN} shows that $\s ^\top \h(x) \s$ can be obtained by twice differentiating function $\lambda \to f(x + \s \lambda)$ at cost of $\tau(\s)$ times of evaluating function $f(x+\s \lambda)$ by using reverse accumulation techniques \citep{christianson1992automatic, gower2012new}. For $\s\sim \cD$ with a constant rank $\tau\eqdef\tau(\s)$, it requires $\cO \left(d\tau^2\right)$ arithmetic operations and in the case $\tau=1$, the cost can be reduced to even $\cO \left(1\right)$ \citep{RSN}.

\subsection{Affine-invarant geometry}
 We can define norms based on a symmetric positive definite matrix $\mathbf H \in \mathbb{R}^{d \times d}$. For any  $x,g \in \Rd,$
\begin{equation}
    \norm{x}_ {\mathbf H} \eqdef  \la \mathbf Hx,x\ra^{1/2}, \qquad \norm{g}_{\mathbf H}^{\ast}\eqdef\la g,\mathbf H^{-1}g\ra^{1/2}.
\end{equation}
As a special case $\mathbf H= \mI$, we get $l_2$ norm $\norm{x}_\mI = \la x,x\ra^{1/2}$. We will be setting $\mathbf H$ to be a Hessian at local point, $\mathbf H = \nabla^2 f(x)$, with the shorthand notation for $g,h \in \Rd:$
\begin{equation} 
{\norm{h}_x \eqdef \la \nabla^2 f(x) h,h\ra^{1/2}, } \, {\norm{g}_{x}^{\ast} \eqdef \la g,\nabla^2 f(x)^{-1}g\ra^{1/2}.}
\end{equation}
The main advantage of the local Hessian norm $\normM h x$ is its affine-invariance, as the affine transformation $f(x) \rightarrow \phi(y) \eqdef f(\mathbf Ay),$ and  $x\rightarrow \mathbf A ^{-1} y$ imply 
\begin{align*}
    \normsM z {\nabla^2 \phi(y)} 
    &=\la\nabla^2 \phi(y)z,z \ra 
    = \la \mathbf A^\top \h(\mathbf A y) \mathbf Az,z\ra =
    \la \h(x)h,h\ra
    =\normsM h{\nabla^2 f(x)}.  
\end{align*}
On the other hand, induced norm $\norm{h}_\mI$ is not, because
\begin{align*}
    \normsM z \mI 
    =\la z,z \ra = \la \mathbf A^{-1}h, \mathbf A^{-1}h\ra
    = \norm{ \mathbf A^{-1} h}^2_\mI.  
\end{align*}
With respect to geometry around point $x$, the more natural norm is the local Hessian norm, $\norm{h}_{\nabla f(x)}$. Affine-invariance implies that its level sets $\lb y\in \R^d\,|\,\norm{ y-x }_x^2 \leq c\rb$ are balls centered around $x$ (all directions have the same scaling). In comparison, the scaling of the $l_2$ norm is dependent on the eigenvalues of the Hessian. In terms of convergence, one direction in $l_2$ can significantly dominate others and slow down an algorithm.
As we are restricting iteration steps to the subspaces, we use shorthand notation {$\normMS h x \eqdef \normM h {\hS(x)}.$}

\begin{table}[t]
    \centering
    \renewcommand\arraystretch{1}
    \setlength\tabcolsep{0.5pt} 
    \begin{threeparttable}[b]
        {
        \scriptsize
            \caption[Three approaches for second-order global minimalization]{
            Three approaches for second-order global minimization. We denote $x^k\in \Rd$ iterates, $\sk \sim \cD$ distribution of sketches of rank $\tau \ll d$, $\als k$ stepsizes.
            We report complexities for matrix inversions implemented naively. }
            \label{tab:three_ways}
            \centering 
            \begin{tabular}{?c?c|c|c?}
            \bottomrule
             \makecell{\textbf{Orthogonal}\\ \textbf{lines of} {\textbf{work}}} 
             & \makecell{\textbf{Sketch-and-project} \\ \citep{gower2015randomized} \\ (various update rules)} 
             & \makecell{\textbf{Damped Newton methods} \\\citep{nesterov1994interior}, \\ \citep{KSJ-Newton2018} } 
             & \makecell{\textbf{Globally Regularized Newton methods} \\ \citep{nesterov2006cubic}, \\\citep{polyak2009regularized},  \citep{mishchenko2021regularized}, \\ \citep{doikov2021optimization} }\\
             \hline
             \makecell{\textbf{Update:}\\ $x^{k+1} - x^k =$ \\ \, } 
             & \makecell{$=\als k \pk k \left( \text{update}(x^k) \right)$}
             & $=\als k [\h(x^k)]^\dagger \g(x^k)$
             & \makecell{$=\argmin_{h\in \Rd} T(x^k, h)$, for $T(x,h)$\\ $ \eqdef \la \g(x), h \ra + \frac 12 \normsM {h} x + \frac {L_2}6 \normM {h} {2} ^3$}\\                         
             \hline
             \makecell{\textbf{Characteristics}}
             &\makecell[l]{
             \color{mydarkgreen}
             + cheap, low-rank updates\\
             \color{mydarkgreen}
             + global linear convergence\\
             \color{mydarkred}
             -- local quadratic rate unachievable\\
             }
             & \makecell[l]{
             \color{mydarkgreen}
             + affine-invariant geometry\\
             \color{mydarkred}
             -- iteration cost $\cO\left(d^3 \right)$\\
             Constant $\als k$:\\
             \color{mydarkgreen}
             \qquad + global linear convergence\\
             Increasing $\als k \nearrow 1$:\\
             \color{mydarkgreen}
             \qquad + local quadratic rate\\
             }             
             & \makecell[l]{
             \color{mydarkgreen}
             + global convex rate $\okd$\\
             \color{mydarkgreen}
             + local quadratic rate\\
             \color{mydarkred}
             -- implicit updates\tn{1}\\
             \color{mydarkred}
             -- iteration cost $\cO\left(d^3 \log \frac 1 \varepsilon \right)$\tn{1}\\
             }\\
             \toprule
             
             \bottomrule
             \makecell{\textbf{Combinations}\\ \textbf{+ retained benefits}} 
             & \makecell{\textbf{Sketch-and-project}} 
             & \makecell{\textbf{Damped Newton methods} } 
             & \makecell{\textbf{Globally regularized Newton methods} }\\
             \Xhline{2\arrayrulewidth}
             \makecell{\rsn{} \\ \citep{RSN} \\ \Cref{alg:rsn}}
             & \makecell{ \cmark\\ \makecell[l]{
             \color{mydarkgreen}
             + iter. cost $\cO\left(d \tau^2 \right)$\\ 
             \color{mydarkgreen}
             + iter. cost $\cO\left(1 \right)$ if $\tau=1$\\ 
             }} & \makecell{\cmark\\
             \color{mydarkgreen}
             + global linear rate 
             \\ }& \xmark\\
             \hline
             \makecell{\sscn{}\\ \citep{hanzely2020stochastic} \\ \Cref{alg:sscn}}
             &\makecell{\cmark\\ \makecell[l]{
             \color{mydarkgreen}
             + iter. cost $\cO\left(d\tau^2 +\tau^3 \log \frac 1 \varepsilon \right)$\\
             \color{mydarkgreen}
             + iter. cost $\cO\left(\log \frac 1 \varepsilon \right)$ if $\tau=1$\\
             \color{mydarkgreen}
             + local linear rate $\cO\left( \frac d \tau \log \frac 1 \varepsilon\right)$}} & \xmark & \makecell{ \cmark \\
             \color{mydarkgreen}
             + global convex rate $\okd$ \\}\\
             \hline
             \makecell{\aicn{}\\ \citep{hanzely2022damped} \\ \Cref{alg:aicn}}
             & \xmark & \makecell{\cmark\\ \makecell[l]{
             \color{mydarkgreen}
             + affine-invariant geometry\\                          
             \color{mydarkred}
             -- no proof of global linear rate\tn{3}\\}} & \makecell{\cmark\\ \makecell[l]{
             \color{mydarkgreen}
             + global convex rate $\okd$\\
             \color{mydarkgreen}
             + local quadratic rate\\
             \color{mydarkgreen}
             + iteration cost $\cO\left(d^3 \right)$\\
             \color{mydarkgreen}
             + simple, explicit updates\\
             }}\\
             \hline
             \makecell{\textbf{\sgn{}}\\ \textbf{(this work)} \\ \Cref{alg:sgn}}
             & \makecell{ \cmark \\ \makecell[l]{
             \color{mydarkgreen}
             + iter. cost $\cO\left(d\tau^2 \right)$\\
             \color{mydarkgreen}
             + iter. cost $\cO\left(1 \right)$ if $\tau=1$\\
             \color{mydarkgreen}
             + local lin. rate $\cO\left( \frac d \tau \log \frac 1 \varepsilon\right)$\tn{2} \hspace{1.5mm}  \\
             }}
             & \makecell{\cmark\\ \makecell[l]{
             \color{mydarkgreen}
             + affine-invariant geometry\\
             \color{mydarkgreen}
             + global linear rate 
             \\}} 
             & \makecell{\cmark\\ \makecell[l]{
             \color{mydarkgreen}
             + global convex rate $\okd$\\
             \color{mydarkgreen}
             + simple, explicit updates\\
             \\}} \\
             \toprule
             
             \bottomrule
             \makecell{\textbf{Three views} \textbf{of} \\\color{mydarkgreen} \textbf{\sgn{}}} & \makecell{ \textbf{Sketch-and-project of}\\\textbf{\color{mydarkgreen} damped Newton method} } & \makecell{\textbf{Damped Newton method} \\ \textbf{\color{mydarkgreen}in sketched subspaces} } & \makecell{\textbf{{Affine-Invariant} Newton method} \\ \textbf{\color{mydarkgreen} in sketched subspaces}}\\
             \Xhline{2\arrayrulewidth}
             \makecell{\textbf{Update}\\ $x^{k+1} - x^k =$ \\ \,} 
             & $=\als k \pk k \color{mydarkgreen} [\h(x^k)]^\dagger \g(x^k)$ 
             & $={\color{mydarkgreen}\als k} {\color{mydarkgreen} \sk } [\nabla_{\color{mydarkgreen}\sk}f(x^k)]^\dagger \nabla_{\color{mydarkgreen} \sk}f(x^k)$
             & \makecell{$={\color{mydarkgreen}\sk} \argmin_{h\in \Rd} T_{\color{mydarkgreen}\sk}(x^k, h)$, \\ for $T_{\color{mydarkgreen} \s} (x,h) \eqdef \la \g(x), {\color{mydarkgreen}\s} h \ra +$\\ \qquad \, \qquad $+ \frac 12 \normsM {{\color{mydarkgreen}\s} h} x + \frac {\Ls}6 \Vert {{\color{mydarkgreen}\s} h} \Vert ^3 _ {\color{mydarkgreen}x} $} \\          
             \toprule
            \end{tabular}
        }
        \begin{tablenotes}
        	{\scriptsize
                \item [\color{red}(1)]Works \citep{polyak2009regularized, mishchenko2021regularized, doikov2021optimization} present algorithms with exact updates and cheaper per-iteration cost, but for the cost of some tradeoffs. In particular, \citet{polyak2009regularized} has slow $\cO(k^{-1/4})$ global rate in convex regime. \citet{mishchenko2021regularized} and \citet{doikov2021optimization} have superlinear, but not quadratic local convergence.
                \item [\color{red}(2)] The rate is independent of the problem conditioning and faster than any first-order method. 
                \item [\color{red}(3)] \citet{hanzely2022damped} didn't show global linear rate of \aicn{}. However, it follows from our Th. \ref{th:global_linear}, \ref{th:local_linear} for deterministic sketches $\sk \equiv \mI$.
        	}
        \end{tablenotes}
    \end{threeparttable}
\end{table}

\section{Algorithm}
\subsection{Three faces of the algorithm}
Our algorithm combines the best of three worlds (\Cref{tab:three_ways}) and we can write it in three different ways.

\begin{theorem} 
\label{th:three}
    If $\g(x^k) \in \Range{\h(x^k)}$\footnote{$\Range{\mathcal A}$ denotes column space of the matrix $\mathcal A$.}, then the update rules are equivalent:
    \begin{align}
    \text{\Rnewton{}:} \qquad &x^{k+1} = x^k + \sk\argmin_{h\in \Rd} \modelSk{x^k, h}, \label{eq:sgn_reg}\\
     \text{\Dnewton{}:} \qquad &x^{k+1} = x^k - \als k \sk [\hSk(x^k)]^\dagger \gSk(x^k), \label{eq:sgn_aicn}\\
     \text{\Sap{}:} \qquad &x^{k+1} = x^k - \als k \pk k [\h(x^k)]^\dagger \g(x^k), \label{eq:sgn_sap}
    \end{align}
    where $\pk k$ is a projection matrix onto $\Range{\sk}$ with respect to norm $\normM \cdot {x_k}$ (defined in eq. \eqref{eq:px}),
    \begin{align}
        \modelS {x,h}
        &\eqdef f(x) + \la \g(x), \s h \ra + \frac 12 \normsM {\s h} x + \frac {\Lalg}6 \normM {\s h} x ^3 \nonumber \\
        &= f(x) + \la \gS(x), h \ra + \frac 12 \normsMS {h} x + \frac {\Lalg}6 \normMS {h} x ^3, \label{eq:sgn_Ts} \\
        &\hspace{-3mm}\als k \eqdef \frac{-1 +\sqrt{1+2 \Lalg \normMSd{\gS(x^k)} {x^k} }}{\Lalg \normMSd {\gS(x^k)} {x^k} }. \label{eq:sgn_alpha}
    \end{align}
    We call this algorithm Sketchy Global Newton, \sgn{}, it is formalized as \Cref{alg:sgn}.
\end{theorem}
Notice that $\als k \in (0,1]$ and limit cases 
$\als k \xrightarrow{\Lalg \normMSdk{\gSk(x^k)} {x^k} \rightarrow 0} 1$ 
and
$\als k \xrightarrow{\Lalg \normMSdk{\gSk(x^k)} {x^k} \rightarrow \infty} 0.$
For \sgn{}, we can easily transition between gradients and model differences 
$h^{k} \eqdef x^{k+1}-x^k$ by identities
\begin{align} \label{eq:transition-primdual}
    h^k &\stackrel{\eqref{eq:sgn_aicn}}= -\als k \sk [\hSk(x^k)]^\dagger \gSk(x^k) \qquad \text{and} \qquad
    \normM {h^k}{x^k} =\als k \normMSdk {\gSk(x^k)}{x^k}.
\end{align}

\subsection{Geometry of sketches}
We use a projection matrix on subspaces $\s$ with respect to local norms $\normM \cdot x$. Denote
\begin{equation} \label{eq:px}
    \p \eqdef \s \left( \st \h (x) \s \right)^\dagger \st \h(x).
\end{equation}

\begin{lemma} \citep{gower2020variance} \label{le:sketch_equiv}
    Matrix $\p$ is a projection onto $\Range{\s}$ with respect to $\normM \cdot x$.    
\end{lemma}

We aim \sgn{} to preserve Newton's direction in expectation. From \eqref{eq:sgn_sap} we can see that this holds as long as $\s \sim \cD$ is such that $\p$ preserves direction in expectation.
\begin{assumption} \label{as:projection_direction}
    Distribution $\cD$ is chosen so that there exists $\tau>0$, such that
    \begin{equation}
        \mathbb E_{\s \sim \cD} \left[ \p \right] = \afrac \tau d \mI.
    \end{equation}
\end{assumption}
\begin{lemma} \label{le:tau_exp}
Assumption~\ref{as:projection_direction} implies $\mathbb E_{\s \sim \cD} \left[ \tau(\s) \right] = \tau$.
\end{lemma}

Assumption~\ref{as:projection_direction} is formulated in the local norms, so it might seem restrictive. To remediate that, in the experiment section, we demonstrate that in practice it can be omitted altogether. Moreover, in \Cref{ssec:sketch_distribution} we argue that such distribution can be constructed from simpler sketch distributions.  

Projection matrix $\p$ from \eqref{eq:px} has contractive properties, as stated by the following lemma.
\begin{lemma} \label{le:projection_contract}
    Projection matrix $\p$ satisfies for any $g,h \in \Rd, g \in \Range {\h(x)}$ inequalities
    \begin{align}
        \hspace{-1.8mm} \normsM {\p h} x & \leq  \normsM {\p h} x + \normsM {(\mI - \p ) h} x = \normsM h x, \label{eq:proj_contr}\\
        \hspace{-1.8mm} \E{\normsM{\p h} x } 
        & =  h ^\top \h(x) \E \p h
        \stackrel{As.\ref{as:projection_direction}} = \frac \tau d \normsM h x, \label{eq:proj_h}\\
        \hspace{-1.8mm} \E{\normsMd{\p^\top g} x }
        & =  g ^\top \E \p [\h(x)]^\dagger g
        \stackrel{As.\ref{as:projection_direction}} = \frac \tau d \normsMd g x, \label{eq:proj_g}\\
        \hspace{-1.8mm} \E{\normM{\p h} x ^3} & \stackrel {\eqref{eq:proj_contr}}\leq \E{ \normM h x \cdot \normsM{\p h} x } 
        \stackrel{\eqref{eq:proj_h}}= \frac \tau d \normM h x ^3.
    \end{align}
\end{lemma}

\subsection{Affine-invariant assumptions}
To leverage the affine-invariance of the norms, we use an affine-invariant version of second-order smoothness called \emph{semi-strong self-concordance} \citep{hanzely2022damped}.
\begin{definition} \label{def:semi-self-concordance}
	A twice differentiable convex $f: \Rd \to \R$  is $\Lsemi-$semi-strongly self-concordant if
	\begin{equation}
		\norm{\nabla^2 f(y)-\nabla^2 f(x)}_{op} \leq \Lsemi \norm{y-x}_{x} , \quad \forall y,x \in \R^d,
	\end{equation} 
    where operator norm is, for given $x \in \R^d$, defined for any matrix $\mathbf H \in \R^{d \times d}$ as 
    \begin{equation} \label{eq:intro_matrix_operator_norm}
        \normM{\mathbf H} {op} \eqdef \sup_{v\in \R^d} \frac {\normMd {\mathbf H v} x}{\normM v x}.
    \end{equation}
\end{definition}
Semi-strong self-concordance differs from the standard second-order smoothness only in the norm in which the distance is measured. It follows from the Lipschitz smoothness and the strong convexity. 

In \Cref{sec:lin_glob} we relax the \Cref{def:semi-self-concordance} to the \emph{self-concordance} \citep{nesterov1994interior}, defined below. To allow tighter constants, we define it separately for each sketched subspace.
\begin{definition} \label{def:scs}
	A three times differentiable convex function $f:\Rd \to \R,$ is called $\Ls$--self-concordant in range of $\s$ if a finite constant $\Ls<\infty$ satisfies
	\begin{equation}
		\left \vert \nabla^3 f(x) [\s h]^3 \right \vert \leq \Ls \normM {\s h} x ^3, \qquad \forall x \in \Rd, h \in \R^{\tau(\s)} \setminus \left \{ 0 \right \},
	\end{equation}
	where 
	$\nabla^3 f(x)[h]^3 \eqdef \nabla^3 f(x) [h, h, h]$ is $3$--rd order directional derivative of $f$ at $x$ along $h \in \R^d$.
\end{definition}

We refer the reader for the more detailed comparison of smoothness assumptions to \Cref{sec:concordances}.

\subsection{One step decrease}
The self-concordance implies a standard decrease in terms of the gradient norms.
\begin{lemma} \label{le:one_step_dec}
    \sgn{} step \eqref{eq:sgn_aicn} decreases loss of $\Ls$--self-concordant function $f:\R^d\to\R$ as
    \begin{align}
        &f(x^k)-f(x^{k+1}) \geq \frac 12 \min \left \{ \left(\Lalg \normMSdk {\gSk(x^k)} {x^k} \right)^{-\frac 12} , \frac 12 \right\} \normsMSdk {\gSk(x^k)} {x^k}.\label{eq:one_step_dec} 
    \end{align}
\end{lemma}
However, decrease \eqref{eq:one_step_dec} is not useful for the fast global $\okd$ rate. Instead, we leverage insights from the regularized Newton methods: that the model $\modelS{x,h}$ upper bounds the function and minimizing it in $h$ decreases the function value.

\begin{proposition}[\citet{hanzely2022damped}, Lemma 2]  \label{pr:semistrong_stoch}
    For $\Lsemi$--semi-strong self-concordant $f:\R^d \to \R$, and any $x \in \Rd, h \in \mathbb{R}^{\tau(\s)}$, sketches $\s \in \mathbb{R}^{d \times \tau(\s)}$ 
    and $x_+ \eqdef x+ \s h$ it holds
    \begin{align}
        \left \vert f(x_+) - f(x) - \la \g(x), \s h \ra - \afrac 1 2 \normsM {\s h} x \right \vert \leq \afrac {\Lsemi} 6 \normM {\s h} x ^3 \label{eq:semistrong_approx}
    \quad \text{and} \quad 
        f(x_+) \leq \modelS{x,h},
    \end{align}
    hence for $h^* \eqdef \argmin_{h \in \mathbb{R}^{\tau(\s)}} \modelS {x,h}$ and corresponding $ x_+ $ we have functional value decrease,
    \begin{align*}
        f(x_+) \leq \modelS{x,h^*} = \min_{h \in \tau(\s)} \modelS {x,h} \leq & \modelS{x,0}= f(x).
    \end{align*}
\end{proposition}

We formulate the one-step decrease lemma analogical to regularized Newton methods. 
\begin{lemma} \label{le:global_step}
    Fix any $y \in \Rd$. Let the function $f:\R^d\to\R$ be $\Lsemi$--semi-strong self-concordant and sketch matrices $\sk \sim \cD$ have unbiased projection matrix, Assumption~\ref{as:projection_direction}. Then \sgn{} has the decrease    
    \begin{align}  \label{eq:global_step}
        \E{f(x^{k+1}) | x^k} &\leq \left(1- \afrac \tau d \right) f(x^k) + \afrac \tau d f(y) + \afrac \tau d\afrac {\max \Lalg + \Lsemi}  6 \normM{y-x^k} {x^k} ^3 .
    \end{align}
\end{lemma}

With \Cref{le:global_step}, we are one step before the main converge result. All that is left is to choose $y$ as a linear combination of $x^k$ and $x^*$ and to bound distance between $\normM {x^k - x^*}{x^k}$.

\section{Main convergence results}
We are ready to present the main convergence results, namely: \textbf{i)} the fast global \textbf{$\okd$} rate, \textbf{ii)} the fast local conditioning-independent linear rate, \textbf{iii)} the global linear convergence rate to the neighborhood of the solution.

\begin{table*}
    \centering
    \setlength\tabcolsep{1.5pt} 
    \begin{threeparttable}[b]
        {\scriptsize
            \renewcommand\arraystretch{2.2}
            \caption[Summary of globally-convergent Newton methods]{             
            Globally convergent Newton-like methods for smooth, strongly convex functions.
            We highlight the best-known rates in blue. }
            \label{tab:detailed}
            \centering 
            \begin{tabular}{cccccccc}\toprule[.2em]
                \bf Algorithm & \bf \makecell{Stepsize \\ range} & \bf  \makecell{Affine invariant \\ algorithm?} & \bf \makecell{Iteration cost\tn{0}} &\bf  \makecell{Linear\tn{1} \\ convergence} & \bf \makecell{Global convex \\convergence } & \bf Reference \\
                \bottomrule[.2em]
                \makecell{\newton{}} & $1$ & \cmark & $\mathcal O \left(d^3 \right)$ & \xmark & \xmark &  \citep{kantorovich1948functional} \\
                \midrule
                \makecell{\Dnewton{} B}  & $(0,1]$ & \cmark & $\mathcal O \left(d^3 \right)$ & \xmark & $\mathcal O \left(k^{-\frac 12} \right)$ &  \citep{nesterov1994interior} \\
                \midrule
                \makecell{\aicn{}} & $(0,1]$ & \cmark & $\mathcal O \left(d^3 \right)$ & \xmark & $\cO\left( {\color{blue} {k^{-2}} } \right)$ &   \citep{hanzely2022damped} \\
                \bottomrule[.2em]
                \makecell{\Cnewton{}} & 1 & \xmark &  $\mathcal O \left(d^3 \log \frac 1 \varepsilon \right)$\tn{3} &  \xmark & $\mathcal O \left( {\color{blue} k^{-2} } \right)$ &  \citep{nesterov2006cubic} \\
                \midrule
                \makecell{\Gnewton{}} & $1$ & \xmark & $\mathcal O \left(d^3 \right)$ & \xmark & $\mathcal O \left(k^{-\frac 14} \right)$ &   \citep{polyak2009regularized} \\
                \midrule
                \makecell{\Gnewton} & 1& \xmark & $\mathcal O \left(d^3 \right)$ & \xmark & $\mathcal O \left( {\color{blue} k^{-2} } \right)$  & \makecell{\citep{mishchenko2021regularized}, \\ \citep{ doikov2021optimization}} \\
                \bottomrule[.2em]
                \makecell{\en{} \\ (\Cref{alg:end})} & $\frac 1 {\Lrel}$\tn{4} & \cmark & $\mathcal O \left(d^3 \right)$ & global\tn{4} & \xmark &  \citep{KSJ-Newton2018} \\
                \midrule
                \makecell{\rsn{}\\(\Cref{alg:rsn})} & $\frac 1 {\Lrel}$\tn{4} & \cmark & \makecell[l]{
                    $\mathcal O({\color{blue}d\tau^2}),$\\
                    $\mathcal O({\color{blue} 1})$ \quad  if $\tau=1$} & global\tn{4} & $\mathcal O \left(k^{-1} \right)$  & \citep{RSN} \\
                \midrule
                \makecell{\sscn{}\\ \Cref{alg:sscn}} & 1 & \xmark & \makecell[l]{
                $\mathcal O(d\tau^2+\tau^3 \log \frac 1 \varepsilon),$\\
                $\mathcal O(\log \frac 1 \varepsilon)$ \quad if $\tau=1$}\tn{3} & local & $\mathcal O \left(k^{-1} \right)$ & \citep{hanzely2020stochastic} \\
                \bottomrule[.2em]
                \makecell{\textbf{\sgn{}}\\(\Cref{alg:sgn})} & $(0,1]$ & \cmark & \makecell[l]{
                    $\mathcal O({\color{blue}d\tau^2}),$\\
                    $\mathcal O({\color{blue} 1})$ \quad  if $\tau=1$} 
                     & \makecell{loc. + glob.\tn{5}} & $\mathcal O \left( {\color{blue} k^{-2} } \right)$ & \textbf{This work} \\
                \bottomrule[.2em]
            \end{tabular}
        }
        \begin{tablenotes}
            {\scriptsize
                \item [\color{red}(0)] Constants $d$ is dimension, $\tau \ll d$ is rank of sketches $\sk $. We report the rate of implementation using matrix inverses.
                \item [\color{red}(1)] Terms ``local" and ``global'' denote whether algorithm has local or global linear rate 
                (under possibly stronger assumptions).
                \item [\color{red}(3)] Cubic Newton and SSCN solve an implicit problem in each iteration, which naively implemented, requires $\times \log  \frac 1 \varepsilon$ matrix inverses approximate sufficiently \citep{hanzely2022damped}. For larger $\tau$ or high precision $\varepsilon$ (case $\tau \log \frac 1 \varepsilon \geq d$), this becomes the bottleneck.
                \item [\color{red}(4)] Relative smoothness constant $\Lrel$ (Def. \ref{def:rel}). \citet{KSJ-Newton2018} shows convergence under a weaker, $c$-stability assumption.
                \item [\color{red}(5)] We have separate results for local convergence (\Cref{th:local_linear}) and global convergence to the corresponding neighborhood (\Cref{th:global_linear}). 
            }
        \end{tablenotes}
    \end{threeparttable}
    \vspace{\vspacefig}
\end{table*}

\subsection{Global convex $\okd$ convergence} \label{sec:conv}
Denote the initial level set and its diameter
\begin{equation}
    \level \eqdef \left\{ x\in\Rd: f(x) \leq f(x^0) \right\} \qquad \text{and} \qquad R \eqdef \sup_{x,y \in \level} \normM {x-y}x.
\end{equation}
Previous lemmas imply that iterates of \sgn{} stay in $\level$, $x^k \in \level \, \forall k\in \N$.

\begin{theorem} \label{th:global_convergence}
    For $\Lsemi$--semi-strongly concordant function $f:\R^d\to\R$ with finite diameter of initial level set $\level$, $R<\infty$ and sketching matrices with Assumption~\ref{as:projection_direction}, \sgn{} has $\cO(k^{-2})$ global convergence rate,
    \begin{align}        
        \E{f(x^k) -\fopt} 
        &\leq \afrac {4 d^3 (f(x^0)-\fopt)}{\tau^3 k^3}  + \afrac {9(\max \Lalg + \Lsemi) d^2 R^3} {2 \tau^2 k^2}.
    \end{align}
\end{theorem}

\subsection{Fast local linear convergence}
Close to the solution, \sgn{} enjoys a conditioning-independent linear rate. This rate is optimal in the sense that as \sap{} method, \sgn{} cannot achieve local superlinear rate (see \Cref{sec:limits}).
\label{sec:lin_loc}
\begin{theorem}\label{th:local_linear}
    Let function $f:\R^d\to\R$ be $\Ls$--self-concordant in subspaces $\s \sim \cD$ and expected projection matrix be unbiased (Assumption~\ref{as:projection_direction}).
    For iterates of \sgn{} $x^0, \dots, x^k$ such that\footnote{It is possible to relax this inequality by replacing $\Lsemi$ by $L_{\sk}.$} $\normMSdk {\gSk(x^k)} {x^k} \leq \frac 1 {\Lsemi}$ and $\g(x^k) \in \Range{\h(x^k)}$, we have local linear convergence
    \begin{equation}
        \E{f(x^{k})-\fopt)} \leq \left( 1- \afrac \tau {bd} \right)^k (f(x^0)-\fopt),
    \end{equation}
    for $b \eqdef \max \left \{\sqrt{\frac \Lalg \Lsemi } ,2 \right\}$,
    and the local complexity of \sgn{} is independent on the problem conditioning, $\mathcal O \left( \sqrt{\frac \Lalg \Lsemi } \frac d \tau \log \frac 1 \varepsilon \right)$ and $\mathcal O \left( \frac d \tau \log \frac 1 \varepsilon \right)$ for $\Lalg = \Lsemi$. 
\end{theorem}

\subsection{Global linear convergence} \label{sec:lin_glob}
Our last convergence result is a global linear rate under relative smoothness and relative convexity.
\begin{definition}\citep{RSN} \label{def:rel}
    We call relative convexity and relative smoothness in subspace $\s$ constants $\murel, \Lrel _\s >0,$ for which for all $\forall x, y \in \mathbb \R^d$ and $y_\s = x+\s h$ for $h \in \mathbb{R}^{\tau(\s)}$ hold:
    \begin{align}
        f(y) &\geq f(x) + \la \g(x), y-x \ra + \afrac {\murel}2 \normsM{y-x} x, \label{eq:rel_conv}\\
        f(y_\s) &\leq f(x) + \la \gS(x), y_\s - x \ra + \afrac {\Lrel _\s} {2}  \normsMS {y_\s - x}{x} \label{eq:rels_smooth}.
    \end{align}
\end{definition}

\citet{RSN} shows that updates $x_+ = y_\s,$ where $y_\s$ is a minimizer of RHS of \eqref{eq:rels_smooth} can be written as Newton method with stepsize $\afrac 1 \Lrel$ and have global linear convergence. Conversely, our stepsize $\als k$ varies, \eqref{eq:sgn_alpha}, so this result is not directly applicable to us. Surprisingly, it turns out that a careful choice of $\Lalg$ can guarantee global linear convergence. Observe following:
\begin{itemize} 
\item We can write \sgn{} model \eqref{eq:sgn_Ts} similarly to the relative smoothness \eqref{eq:rels_smooth},
\begin{align}
    x_+ = x + \s &\argmin_{h \in \mathbb{R}^{\tau(\s)}} \Bigg(f(x) +\la \gS(x), h \ra  +\afrac 12 \left( 1 + \afrac {\Lalg}3 \normMS {h} x \right)  \normsMS {h} x \Bigg). 
\end{align}
If $\left( 1 + \afrac {\Lalg}3 \normMS {h} x \right) \geq \Lrel _\s$, then \sgn{} model upper bounds the right-hand side of \eqref{eq:rels_smooth} and therefore, function $f$ as well. Consequently, we can obtain rates similar to \citet{RSN}.

\item To guarantee $\left( 1 + \afrac {\Lalg}3 \normMS {h} x \right) \geq \Lrel _\s$, we can express $\Lalg$ using $\normMS {h} x = \alpha \normMSd {g} x$ as:
\begin{align}
    1 + \afrac {\Lalg}3 \normMS {h} x \geq  \Lrel _\s 
    &\, \Leftrightarrow \, \Lalg \geq  \afrac {3(\Lrel _\s-1)}{\alpha \normMSd {\gS(x)} x}
    \, \Leftrightarrow \, 1 \geq  \afrac {3 (\Lrel _\s-1)}{-1 + \sqrt{1+2\Lalg \normMSd {\gS(x)} x}} \nonumber\\
    & \, \Leftrightarrow \, \Lalg \geq \afrac 32 \afrac  {(\Lrel _\s -1) (3 \Lrel _\s -1)}{\normMSd {\gS(x)} x}. 
\end{align}
\item We have already shown the fast local convergence of \sgn{} (\Cref{th:local_linear}). Now we need to obtain linear rate for just points $x^k$ beyond that neighborhood of convergence, $\normMSdk {\gSk(x^k)} {x^k} \geq \frac 1 {L_{\sk}}$.

For those points $x^k$ we can guarantee it with the choice $\Lalg \geq \sup_{\s} \afrac 92 \Ls \Lrel _\s ^2$.
Such $\Lalg$ also bounds the stepsize far from the solution as $\als k \Lrels \leq \frac 23 $ (see \Cref{le:stepsize_bound} in \Cref{ssec:towards_glob}).
\end{itemize}

We are almost ready to present the global linear convergence result. Finally, the rate depends on the conditioning of the expected projection matrix $\p$, defined as\footnote{We formulate the condition number $\rho(x)$ in local norms, but $l_2$ norms can be used as well.}
    \begin{align}
        \rho(x) 
        &= \min_{g \in \Rd} \afrac { g^\top \E{ \alpha \p}[\h(x)]^\dagger g} {\normsMd g x} \qquad \text{and} \qquad
        \rho \eqdef \min_{x \in \level} \rho(x). \label{eq:rho}
    \end{align}
\begin{theorem} \label{th:global_linear}
    Let $f:\R^d \to \R$ be $\Lrel_\s$--relative smooth in subspaces $\s$ and $\murel$--relative convex. Let sampling $\s\sim \cD$ 
    satisfy $\Null{\s^\top \h(x) \s} = \Null{\s}$
    and $\Range{\h(x)} \subset \Range {\bbE_{\s \sim \cD}\left[\s \s ^\top\right] }$.     
    Then $0 < \rho \leq 1$.
    Choose parameter $\Lalg = \sup_{\s \sim \cD} \afrac 92 \Ls \Lrel _\s^2$.
    
    While iterates $x^0, \dots, x^k$ satisfy $\normMSdk {\gSk(x^k)} {x^k} \geq \frac 1 {L_{\sk}}$, then\footnote{If this inequality does not hold in the current iterate, \Cref{th:local_linear} guarantees even faster convergence.} \sgn{} has the decrease
    \begin{equation}        
    \E{f(x^{k}) - \fopt}  \leq \left(1 - \afrac 43 \rho \murel \right)^k (f(x^0) - \fopt), 
    \end{equation}
    and global linear $\cO\left( \afrac 1 {\rho \murel} \log \afrac 1 \varepsilon \right)$ convergence.
\end{theorem}

\section{Experiments}
We support our theory by comparing \sgn{} to \sscn{}.
To match practical considerations of \sscn{} and for the sake of simplicity, we adjust \sgn{} in unfavorable way:
\begin{enumerate}      
    \item We choose sketching matrices $\s$ to be unbiased in $l_2$ norms (instead of local hessian norms $\normM \cdot x$ from Assumption~\ref{as:projection_direction}),
    \item To disregard implementation specifics, we report iterations on the $x$--axis. 
Note that \sscn{} needs to use a subsolver (extra line-search) to solve the implicit step in each iteration. If naively implemented using matrix inverses, iterations of \sscn{} are $\times \log \frac 1 \varepsilon$ slower.
We haven't reported time as this would naturally ask for optimized implementations and experiments on a larger scale -- this was out of the scope of the paper.
\end{enumerate}
Despite the simplicity of \sgn{} and unfavorable adjustments, \Cref{fig:vssscn} shows that \sgn{} performs comparably to \sscn{}.
We can point out other properties of \sgn{} based on experiments in the literature.

\begin{itemize}
 
\item \textbf{Rank of $\s$ and first-order methods:} 
\citet{RSN} showed a detailed comparison of the effect of various ranks of $\s$. Also, \citet{RSN} showed that \rsn{} (the Newton method with the fixed stepsize) is much faster than first-order Accelerated Coordinate Descent (\acd{}) for highly dense problems. For extremely sparse problems, \acd{} has competitive performance. As the stepsize of \sgn{} increases as converging to the solution, we expect similar, if not better results.

\item \textbf{Various sketch distributions}:  
\citet{hanzely2020stochastic} considered various distributions of sketch matrices $\s \sim \cD$. In all of their examples, \sscn{} outperformed \cd{} with uniform or importance sampling and was competitive with \acd{}. As \sgn{} is competitive to \sscn{}, similar results should hold for \sgn{} as well.

\item \textbf{Local norms vs $l_2$ norms:} 
\citet{hanzely2022damped} shows that the optimized implementation of \aicn{}  saves time in each iteration over the optimized implementation of \cnewton{}. As \sgn{} and \sscn{} use analogical updates (in the subspaces), it indicates that \sgn{} saves time over \sscn{}.
\end{itemize}

\begin{figure*}
    \centering
    \includegraphics[width=\figsize]{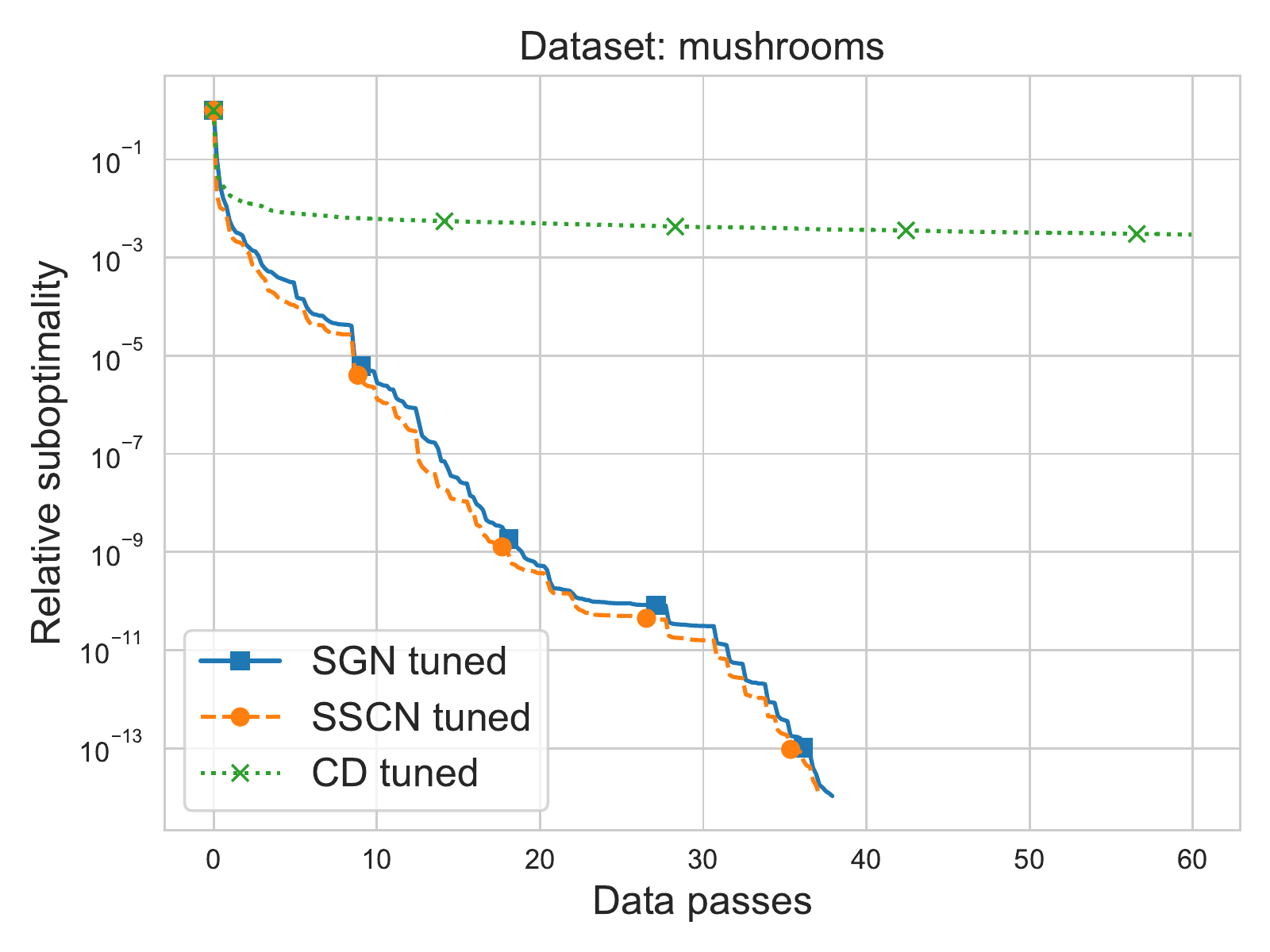}
    \includegraphics[width=\figsize]{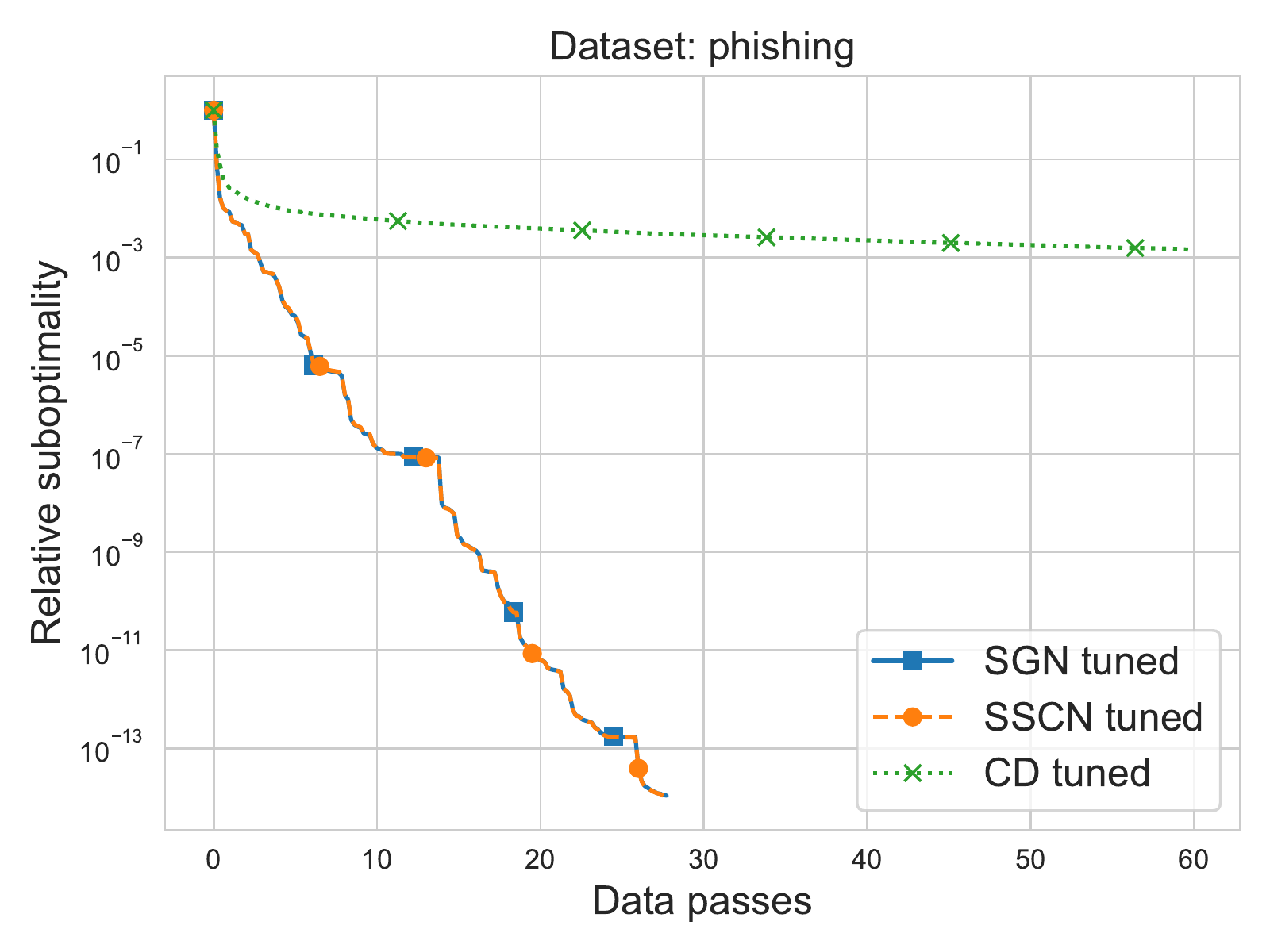}
    \includegraphics[width=\figsize]{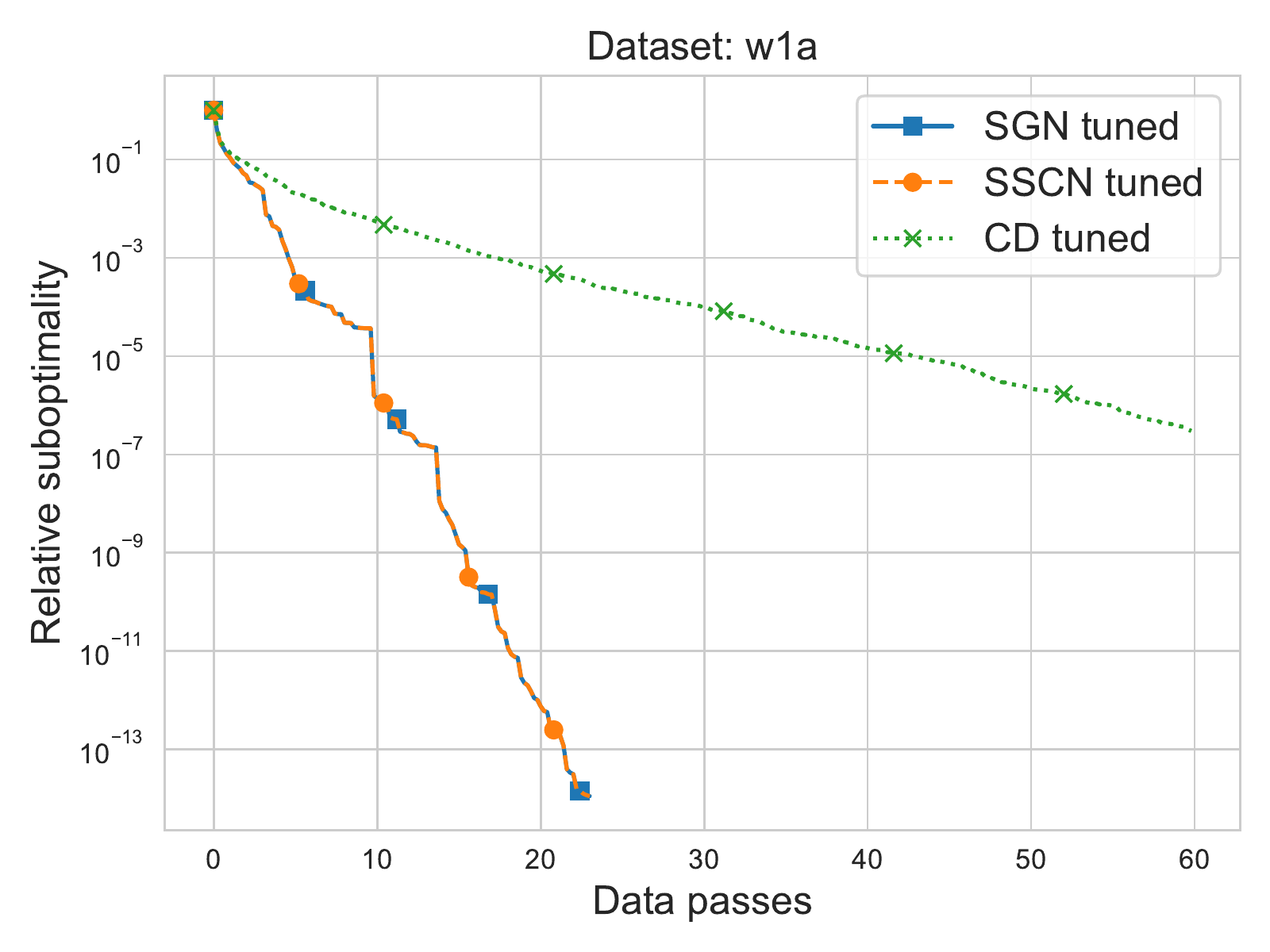}
    \includegraphics[width=\figsize]{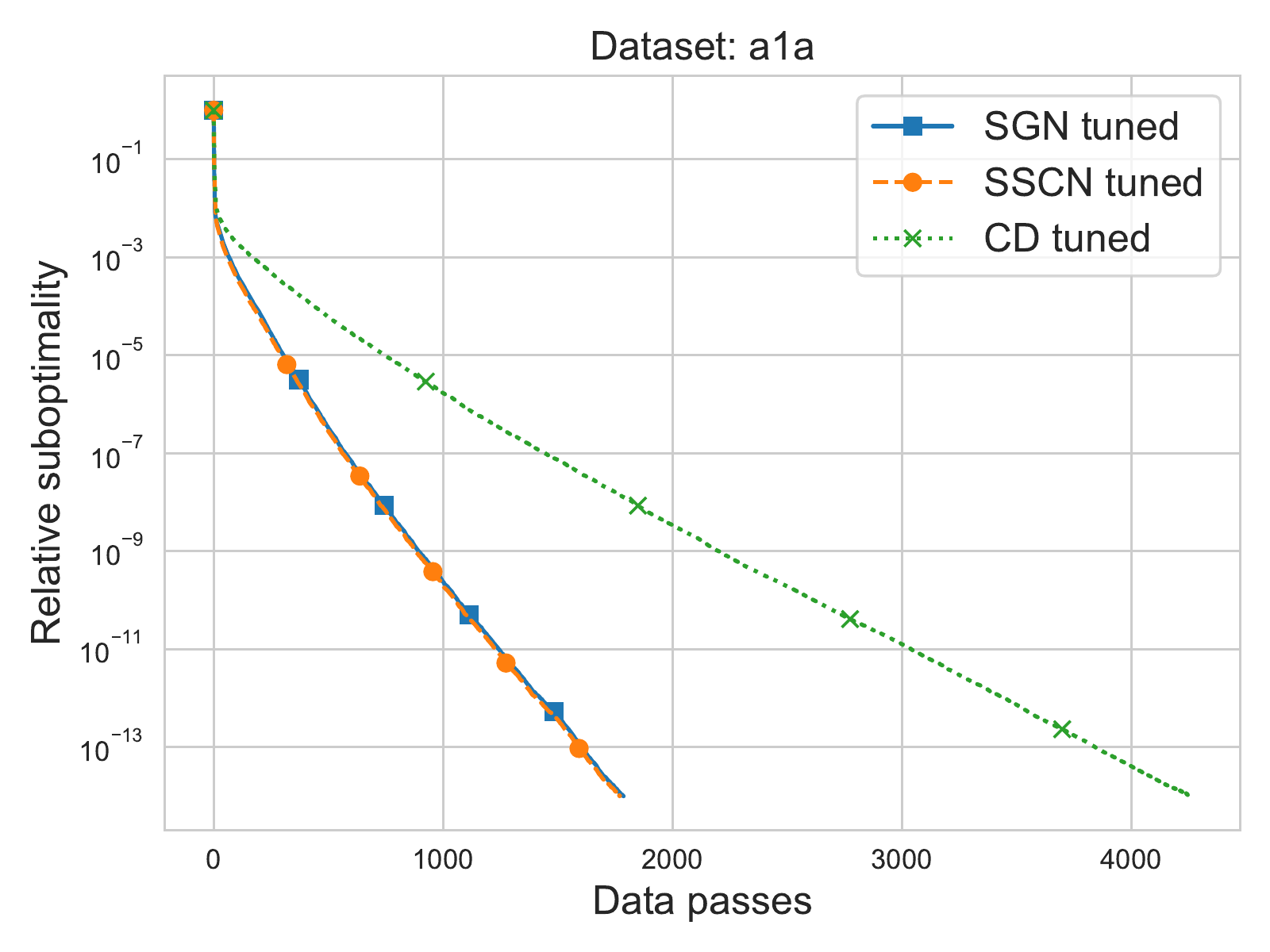}
    \includegraphics[width=\figsize]{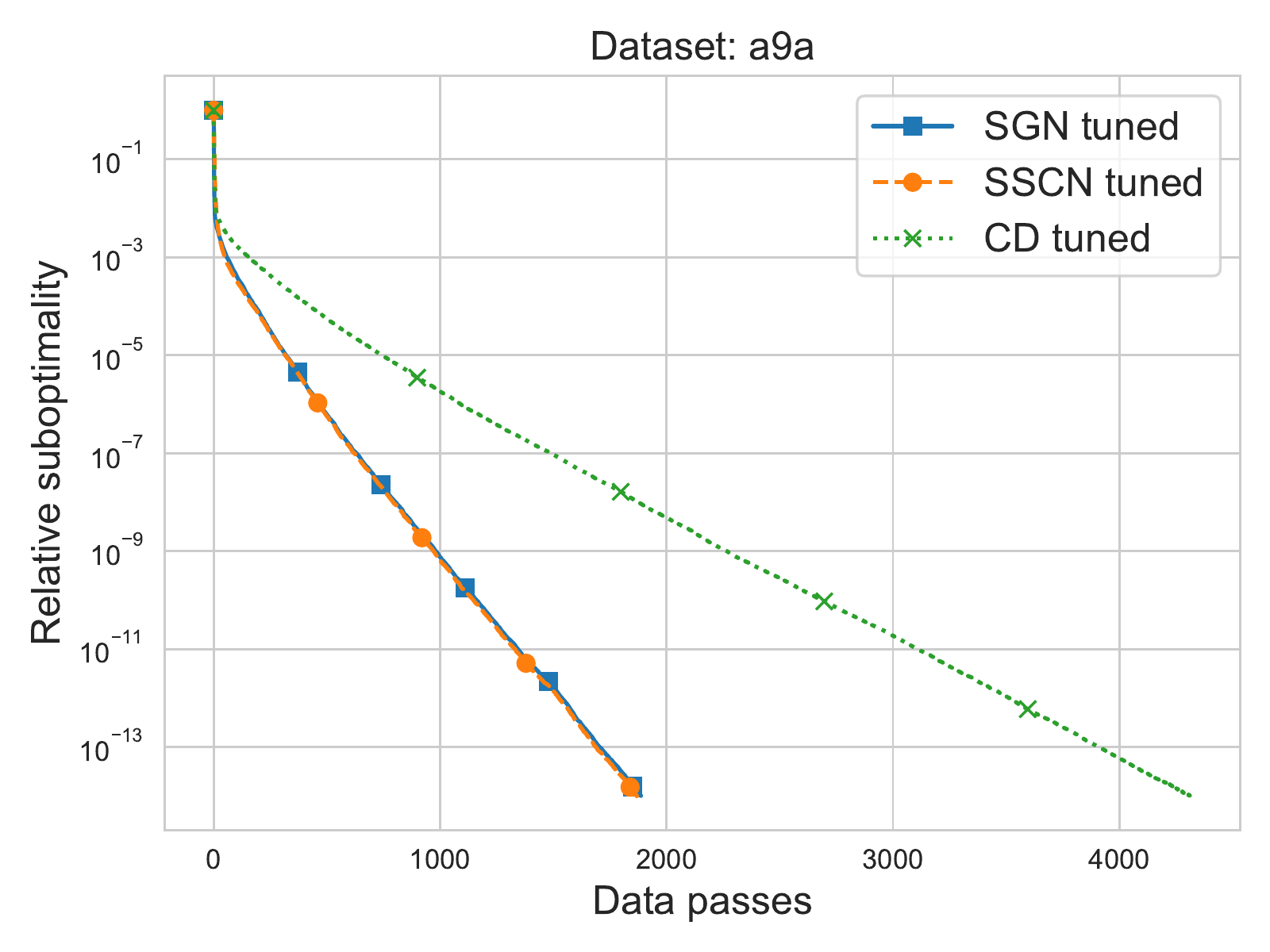}
    \includegraphics[width=\figsize]{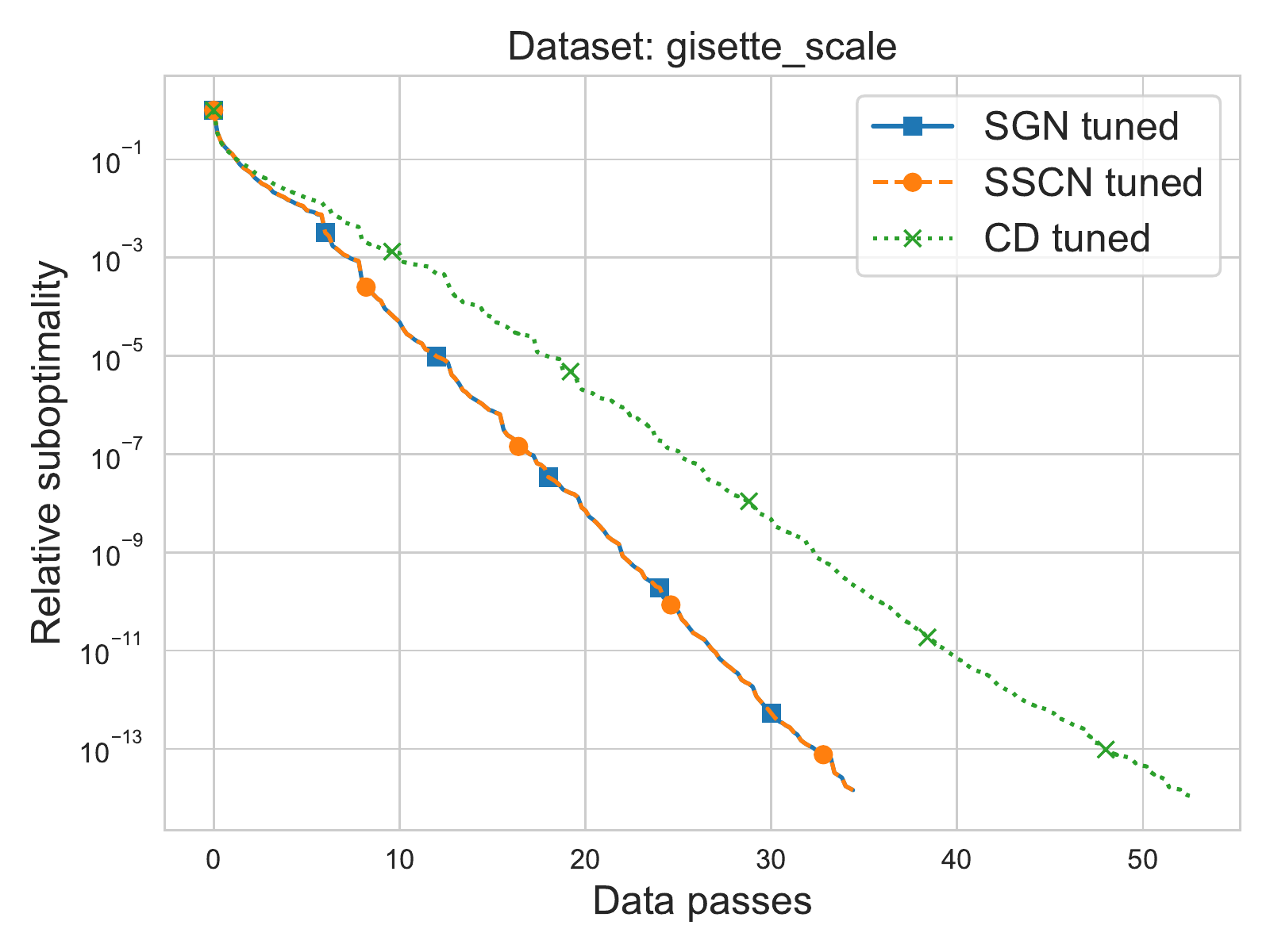}
    \includegraphics[width=\figsize]{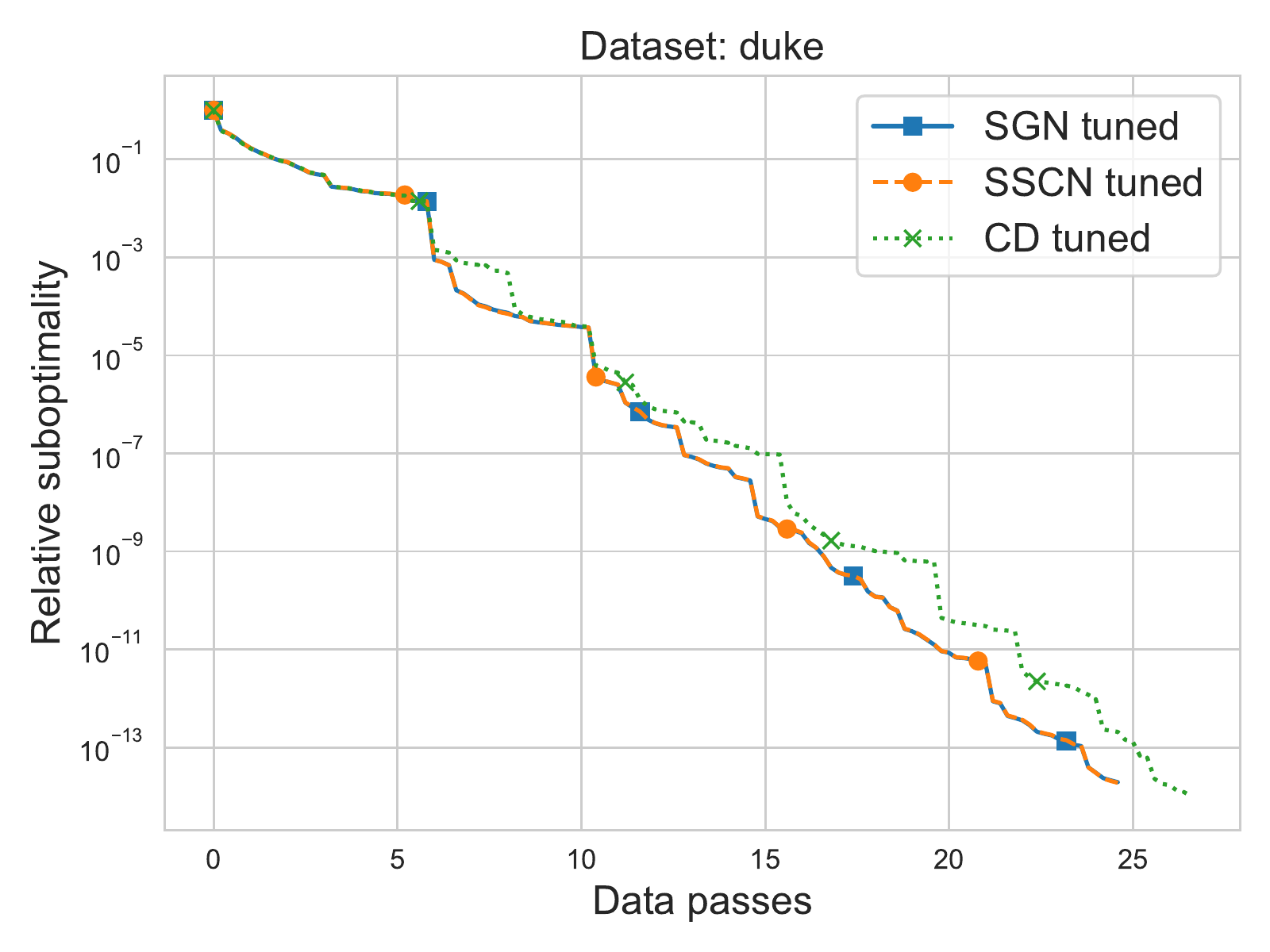}
    \includegraphics[width=\figsize]{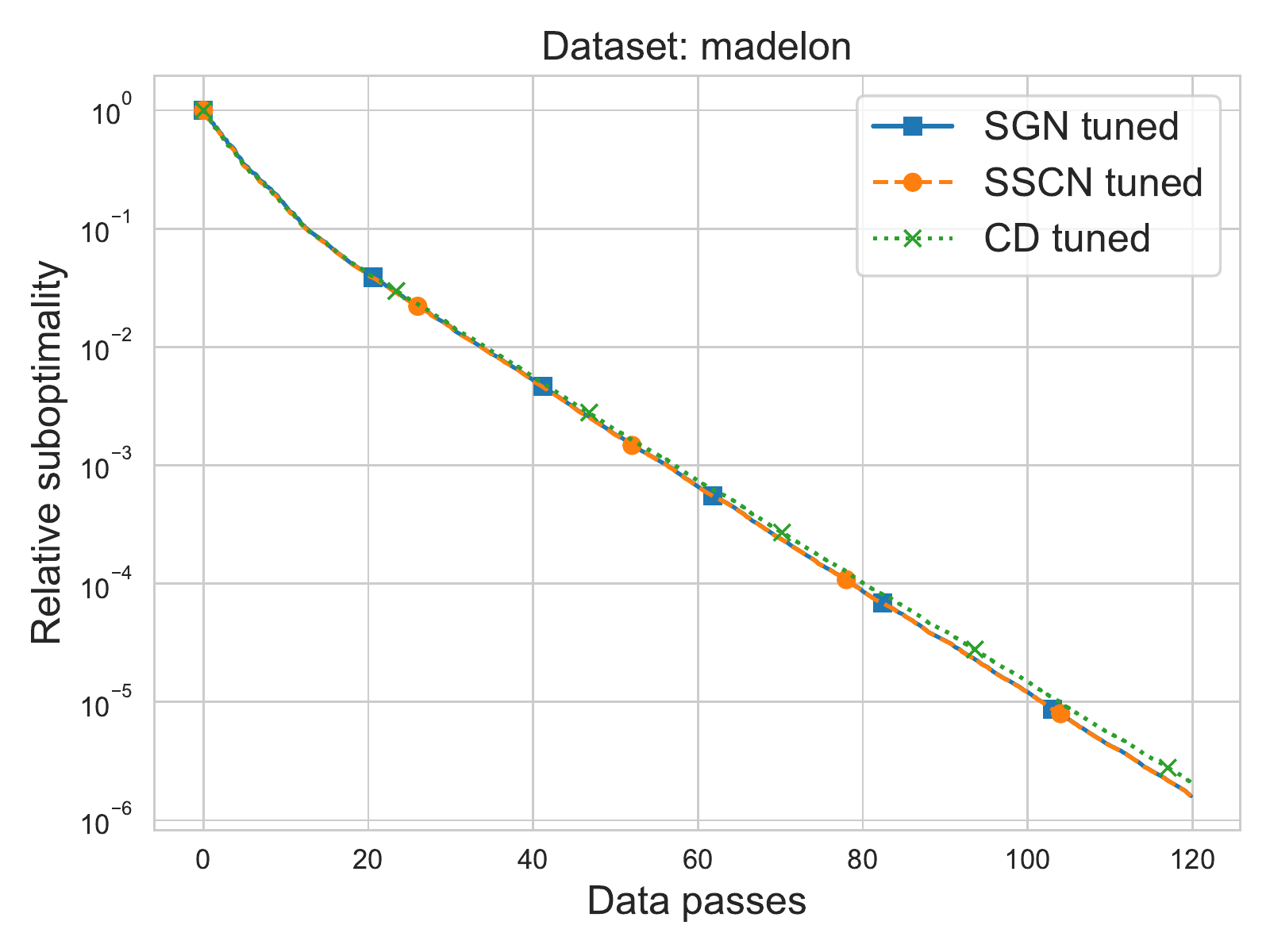}
    \caption[Comparison of low-rank second-order methods with fast convergence guarantees]{Comparison of \sscn{}, \sgn{} and \cd{} on the logistic regression loss on \libsvm{} datasets for sketch matrices $\s$ of rank one. We fine-tune all algorithms for their smoothness parameters.}
    \label{fig:vssscn}
    \vspace{\vspacefig}
\end{figure*}

\bibliographystyle{plainnat}
\bibliography{references}

\clearpage	
\onecolumn
\part*{Appendix} \label{p:appendix}
	\appendix

\section{Table of frequently used notation}
\begin{table}[h!]
	\centering
    \begin{threeparttable}        
	\caption[Summary of frequently used notation in \Cref{sec:sgn}]{Summary of frequently used notation}
	\label{tab:notation}
	\begin{tabular}{|c|l|}
        \hline 
        \multicolumn{2}{|c|}{{\bf General} }\\
        \hline
		$\mathbf A^\dagger$ & Moore-Penrose pseudoinverse of $\mathbf A$\\
		$\normM{\cdot}{op} $ & Operator norm\\
		$\normM{\cdot}x $ & Local norm at $x$\\
		$\normMd{\cdot}x $ & Local dual norm at $x$\\
		$x, x_+, x^k \in \Rd$ & Iterates\\
		$y \in \Rd$ & Virtual iterate (for analysis only)\\
		$h, h' \in \Rd$ & Difference between consecutive iterates\\
		$\als k$ & \sgn{} Stepsize\\
		\hline
        \multicolumn{2}{|c|}{{\bf Function specific} }\\
        \hline
		$d$ & Dimension of problem\\
		$f: \Rd \rightarrow \mathbb{R}$ & Loss function \\
		$\modelS {\cdot, x} $ & Upperbound on $f$ based on gradient and Hessian in $x$\\
		$\opt$, $\fopt$ & Optimal model and optimal function value\\
		$\level$ & Set of models with a functional value less than $x^0$\\
		$R, D, D_2$ & Diameter of $\level$\\
		$\Lstandard, \Lsemi$ & Self-concordance and semi-strong self-concordance constants\\
		$\Lalg$ & Smoothness estimate, affects stepsize of \sgn{}\\
		$\Lrel, \murel$ & Relative smoothness and relative convexity constants\\
		\hline
        \multicolumn{2}{|c|}{{\bf Sketching} }\\
        \hline
		$\gS, \hS, \normMS h x$  & Gradient, Hessian, local norm in range $\s$, resp.\\
		$\s \in \mathbb{R}^{d \times \tau(\s)}$ & Randomized sketching matrix\\
		$\tau(\s)$ & Dimension of randomized sketching matrix\\
		$\tau$ & Fixed dimension constraint on $\s$\\
		$\Ls$ & Self-concordance constant in range of $\s$\\
		$\p$ & Projection matrix on subspace $\s$ w.r.t. local norm at $x$\\         
		$\rho(x)$ & Condition numbers of expected scaled projection matrix $\E{\alpha \p}$\\
		$\rho$ & Lower bound on condition numbers $\rho(x)$\\
        \hline
	\end{tabular}
    \end{threeparttable}
\end{table}

\section{Self-concordance overview} \label{sec:concordances}
Self-concordance \citep{nesterov1994interior} is a variant of the smoothness assumption expressed in local norms.
\begin{definition} \label{def:self-concordance}
	Convex function $f$ with continuous first, second, and third derivatives is called \textit{self-concordant}  if
	\begin{equation}
		|D^3 f(x)[h]^3| \leq \Lstandard\norm{h}_{x}^3, \quad \forall x,h\in \R^d,
	\end{equation}
	where for any integer $p\geq 1$, by $D^p f(x)[h]^p \eqdef D^p f(x)[h,\ldots,h]$ we denote the $p$--th order directional derivative\footnote{For example, $D^1 f(x)[h] =\langle\nabla  f(x),h\rangle$ and $D^2 f(x)[h]^2 = \la \nabla^2 f(x) h, h \ra $.} of $f$ at $x\in \R^d$ along direction $h\in \R^d$.
\end{definition}
This assumption corresponds to a big class of optimization methods called interior-point methods \citep{nesterov1994interior}, it implies the uniqueness of the solution of the lower bounded function \citep[Theorem 5.1.16]{nesterov2018lectures}.

Similarly to Lipschitz smoothness, self-concordance implies an upper bound on the function value; however, in local norms.
\begin{equation} 
	f(y) - f(x) \leq \ip{\g(x)} {y-x} + \frac1 2 \normM {y-x} x ^2 + \frac\Lstandard 6 \normM {y-x} x ^3, \qquad \forall x,y \in \R^d.
\end{equation}    	

Note that this definition matches the definition of self-concordance in sketched subspaces (\Cref{def:scs}) with $\s = \mI$, and consequently $\Ls \leq \Lstandard$. 

Going beyond self-concordance, \citet{rodomanov2021greedy} introduced a stronger version of the self-concordance assumption.
\begin{definition}	\label{def:strong-self-concordance}
	Twice differentiable convex function, $f\in C^2$,  is called \textit{strongly self-concordant} if
	\begin{equation}
		\label{eq:strong-self-concordance}
		\nabla^2 f(y)-\nabla^2 f(x) \preceq \Lstrongly \norm{y-x}_{z} \nabla^2 f(w), \quad \forall y,x,z,w \in \R^d .
	\end{equation} 
\end{definition}

In this paper, we are working with the class of semi-strong self-concordant functions \citep{hanzely2022damped}
\begin{equation}
    \norm{\nabla^2 f(y)-\nabla^2 f(x)}_{op} \leq \Lsemi \norm{y-x}_{x} , \quad \forall y,x \in \R^d,
\end{equation}
which is analogous to standard second-order smoothness
\begin{equation} \label{eq:intro_L2-smooth}
    \norm{ \nabla^2 f(x) - \nabla^2 f(y)} \leq L_2\norm{x-y}.
\end{equation}

All of the mentioned self-concordance variants are affine-invariant and their respective classes satisfy \citep{hanzely2022damped}
\begin{gather*}
	\textit{strong self-concordance} \subseteq \textit{semi-strong self-concordance}
	\subseteq \textit{self-concordance}.
\end{gather*}
Also, for a fixed strongly self-concordant function $f$ and smallest such  $\Lstandard, \Lsemi, \Lstrongly$ holds $\Lstandard \leq \Lsemi \leq \Lstrongly$ \citep{hanzely2022damped}.

All notions of self-concordance are closely related to the standard convexity and smoothness; \citet{rodomanov2021greedy} shows that strong self-concordance follows from function $L_2$--Lipschitz continuous Hessian and strong convexity.
\begin{proposition} \citep[Example 4.1]{rodomanov2021greedy} \label{le:sscf}
	Let $\mathbf H: \R^d \rightarrow \R^d$ be a self-adjoint positive definite operator. Suppose there exist $\mu>0$ and $L_2 \geq 0$ such that the function $f$ is $\mu$--strongly convex and its Hessian is $L_2$--Lipschitz continuous \eqref{eq:intro_L2-smooth} with respect to the norm $\norm{\cdot}_{\mathbf H}$. Then $f$ is strongly self-concordant with constant $\Lstrongly=\frac{L_2}{\mu^{3 / 2}}$. 
\end{proposition}

\section{Additional experiments \& implementation details}

In \Cref{fig:vsacd} we compare \sgn{} and Accelerated Coordinate Descent on small-scale experiments.

We use a comparison framework from \citep{hanzely2020stochastic}, including implementations of \sscn{}, Coordinate Descent, and Accelerated Coordinate Descent.

Experiments are implemented in Python 3.6.9 and run on a workstation with 48 CPUs Intel(R) Xeon(R) Gold 6246 CPU @ 3.30GHz. Total training time was less than $10$ hours. Source code and instructions are included in supplementary materials. As we fixed a random seed, therefore experiments are fully reproducible.

\begin{figure*}
	\centering
	\includegraphics[width=\figsizeap]{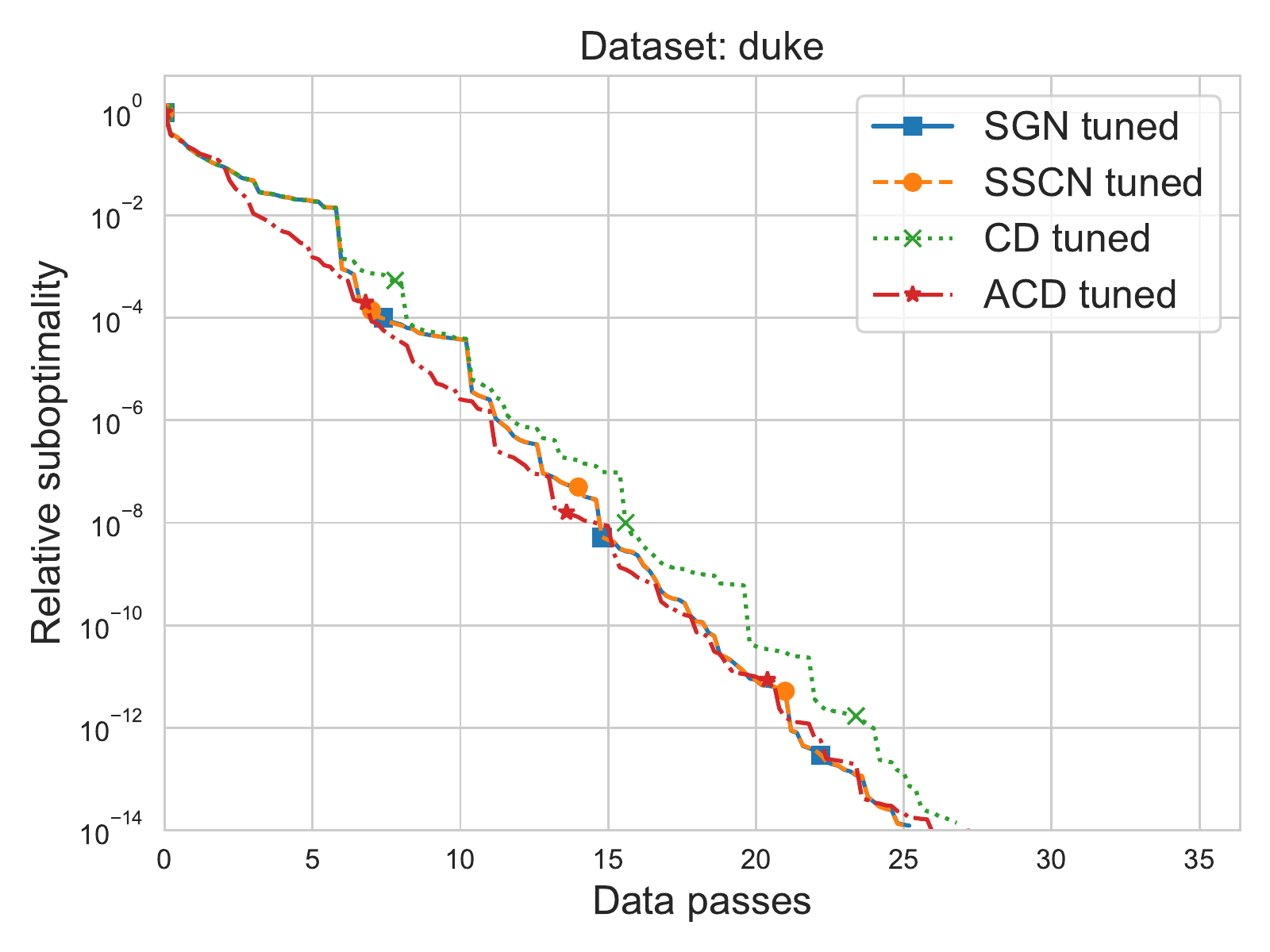}   
	\includegraphics[width=\figsizeap]{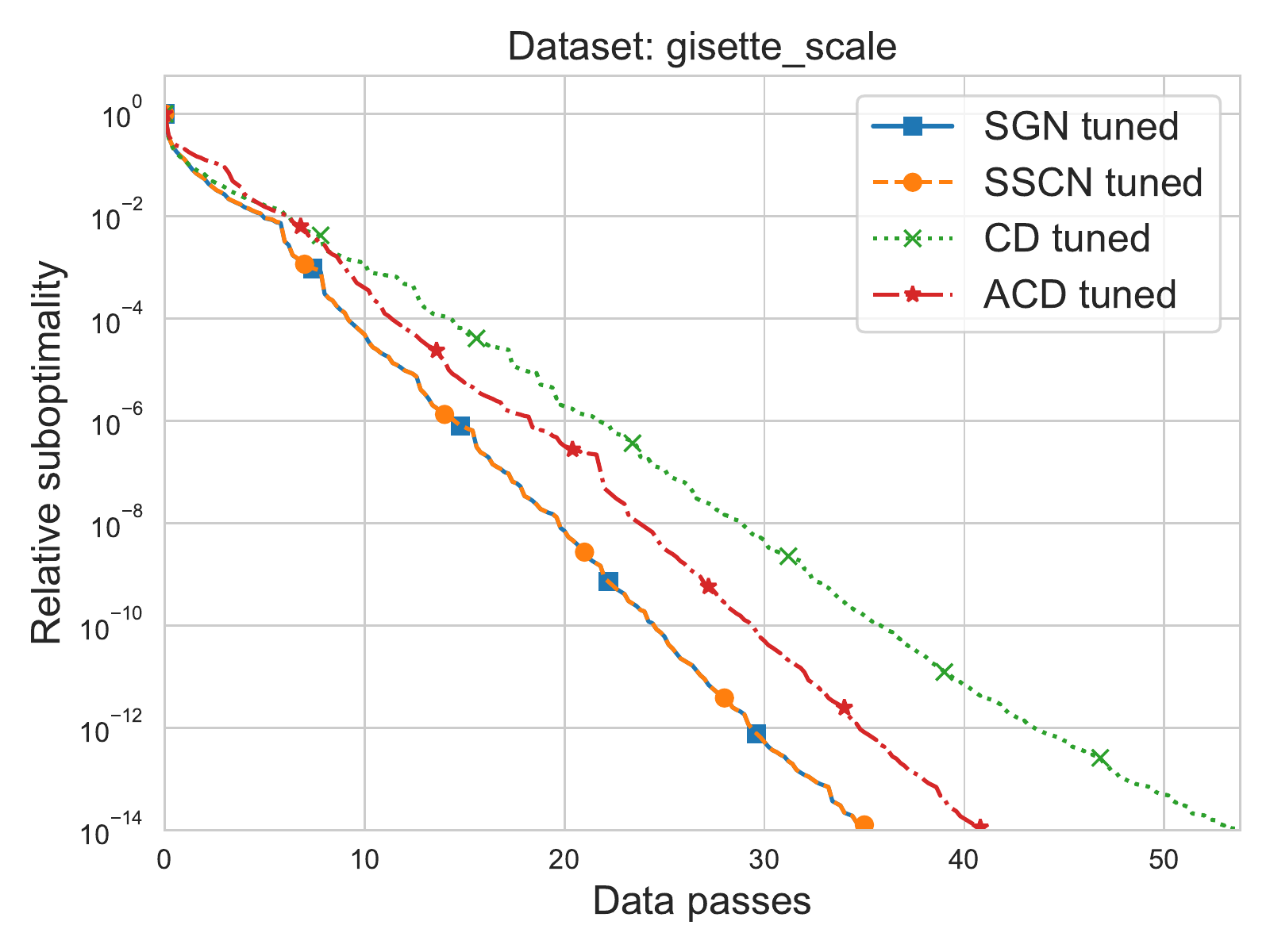}   
	\includegraphics[width=\figsizeap]{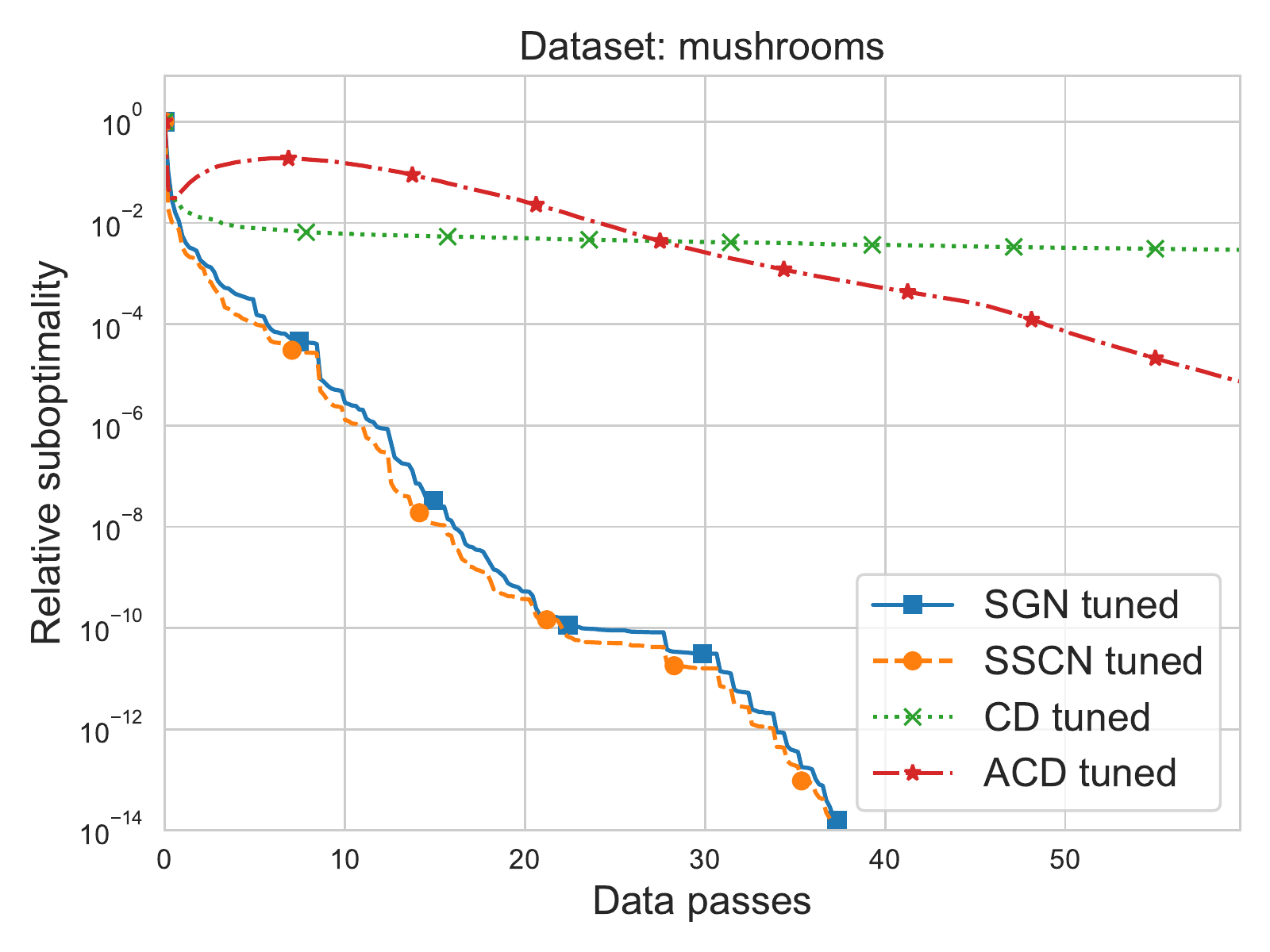}   
	\includegraphics[width=\figsizeap]{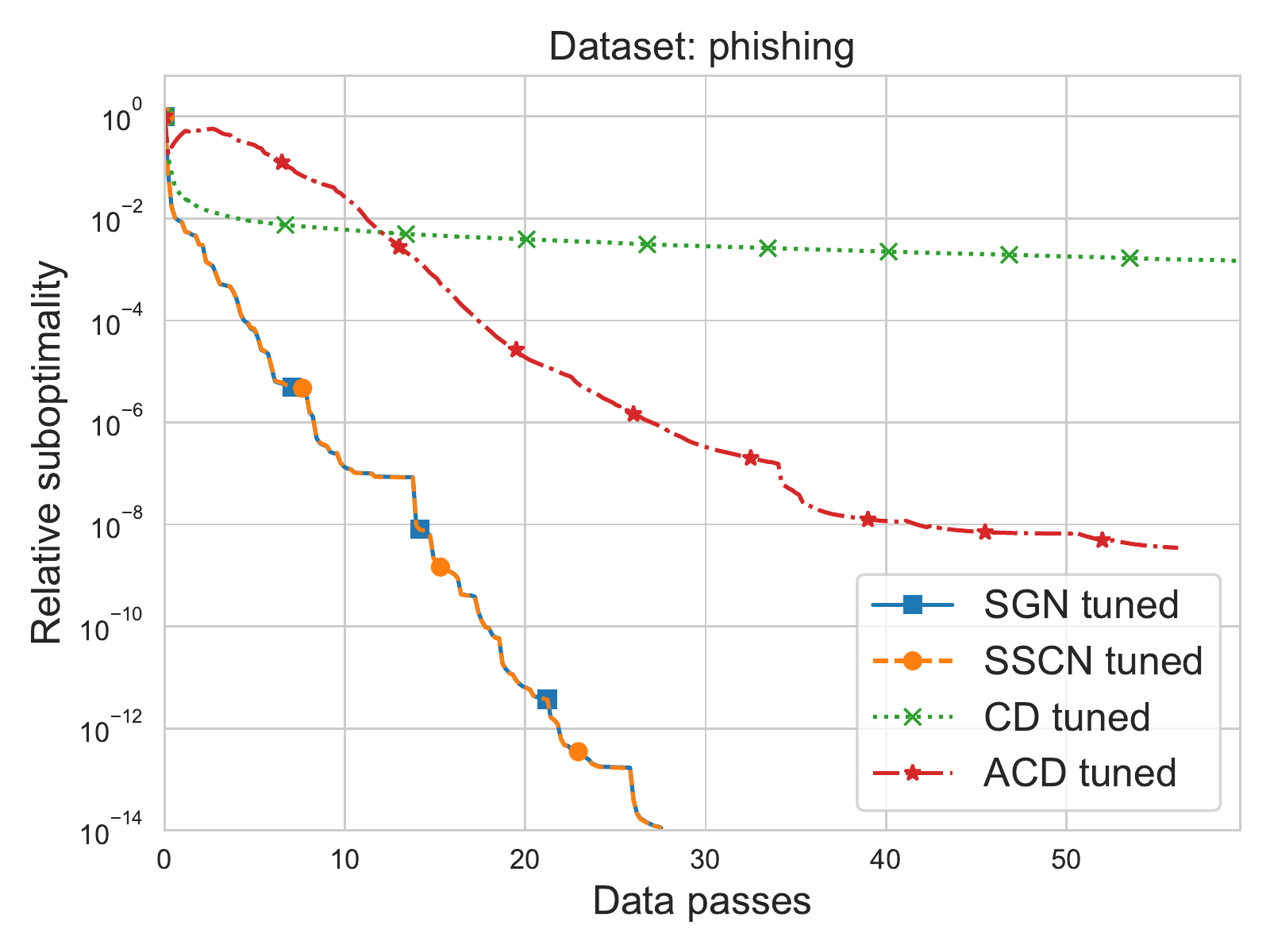}   
	\includegraphics[width=\figsizeap]{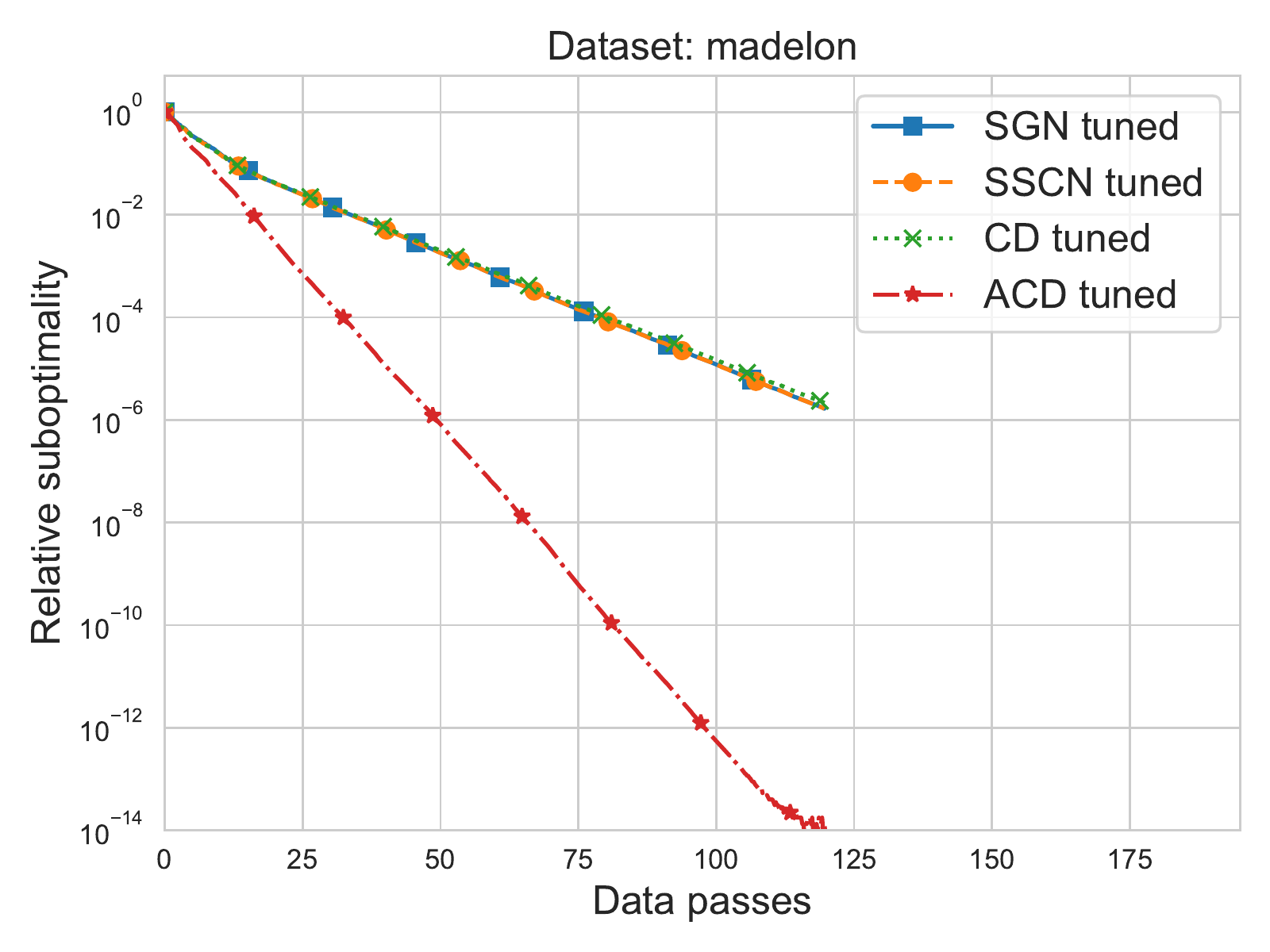}   
	\caption[Comparison of \sgn{} and first-order methods]{Comparison of \sscn{}, \sgn{}, \cd{} and \acd{} on logistic regression on \libsvm{} datasets for sketch matrices $\s$ of rank one. We fine-tune all algorithms for smoothness parameters.}
	\label{fig:vsacd}
    \vspace{\vspacefig}
\end{figure*}

\section{Local linear convergence limit} \label{sec:limits}
Similarly to \aicn{} \citep{hanzely2022damped}, we can show that one step decreases the gradient norm quadratically. In our case, the quadratic decrease is the sketched subspace.
\begin{lemma} \label{le:local_step_subspace}
    For $\Lsemi$--semi-strong self-concordant function $f:\R^d\to\R$ and parameter choice $\Lalg \geq\Lsemi$, one step of \sgn{} has quadratic decrease in $\Range{\sk}$,
    \begin{equation} \label{eq:local_step_subspace}
        \normMSdk {\gSk(x^{k+1})} {x^k} \leq \Lalg \als k^2\normsMSdk{\gSk(x^k) }{x^k}.
    \end{equation}
\end{lemma}
Nevertheless, this is insufficient for superlinear local convergence; we can achieve a linear rate at best.
We can illustrate this on an edge case where $f:\R^d \to \R$ is a quadratic function: self-concordance assumption holds with $\Ls=0$ and as $\als k \xrightarrow {\Ls \rightarrow 0} 1$, \sgn{} stepsize becomes $1$ and \sgn{} simplifies to subspace Newton method. Unfortunately, the subspace \newton{} method has just linear local convergence \citep{RSN}.

\section{Algorithm comparisons}
For readers convenience, we include pseudocodes of the most relevant baseline algorithms: Exact Newton Descent (\Cref{alg:end}), \rsn{} (\Cref{alg:rsn}), \sscn{} (\Cref{alg:sscn}), \aicn{} (\Cref{alg:aicn}).

\begin{figure}[h]
	\begin{minipage}[t]{0.46\textwidth}
		\begin{algorithm}[H]
			\caption{Exact Newton Descent \citep{KSJ-Newton2018}} \label{alg:end}
			\begin{algorithmic}
				\State \textbf{Requires:} Initial point $x^0 \in \Rd$, $c$--stability bound $\sigma >c >0$
				\For {$k=0,1,2\dots$}
				\vspace{4.2mm}
				\State $x^{k+1} = x^k - {\color{blue}\frac 1 {\sigma}} \left[\h(x^k)\right]^{\dagger} \g(x^k)$
				\EndFor
			\end{algorithmic}
		\end{algorithm}
	\end{minipage}
	\hfill
	\begin{minipage}[t]{0.53\textwidth}
		\begin{algorithm}[H]
			\caption{\rsn{}: Randomized Subspace Newton \citep{RSN}} \label{alg:rsn}
			\begin{algorithmic}
				\State \textbf{Requires:} Initial point $x^0 \in \Rd$, distribution of sketches $\cD$, relative smoothness constant $L_{\text{rel}} >0$
				\For {$k=0,1,2\dots$}
				\State {\color{mydarkgreen}Sample $\s_k \sim \cD$}
				\State $x^{k+1} = x^k - {\color{blue} \frac 1 {\Lrel}} {\color{mydarkgreen}\sk} \left[\nabla^2_{{\color{mydarkgreen}\sk}}f(x^k)\right]^{\dagger} \nabla_{{\color{mydarkgreen}\sk}}f(x^k)$
				\EndFor
			\end{algorithmic}
		\end{algorithm}
	\end{minipage}
	\begin{minipage}[t]{0.50\textwidth}
		\begin{algorithm} [H]
			\caption{\aicn{}: Affine-Invariant Cubic Newton \citep{hanzely2022damped}} \label{alg:aicn}
			\begin{algorithmic}
				\State \textbf{Requires:} Initial point $x^0 \in \Rd$, estimate of semi-strong self-concordance $\Lalg \geq \Lsemi>0$    
				\State \,
				\For {$k=0,1,2\dots$}
				\State {\color{blue}$\als k = \frac {-1 +\sqrt{1+2 \Lalg \norm{\g(x^k)}_{x^k}^* }}{\Lalg \norm{ \g(x^k)}_{x^k}^* }$}
				\State $x^{k+1} = x^k - {\color{blue}\als k} \left[\h(x^k)\right]^{-1} \g(x^k)$\footnote {Equival.,  $x^{k+1} = x^k - {\color{brown}\argmin_{h \in \Rd} T(x^k, h)}$, \\ for {\color{brown}$T(x,h) \eqdef \la \g(x), h \ra + \frac 12 \normsM h x + \frac {\Lalg}6 \normM h x ^3$}.}
				\EndFor
			\end{algorithmic}
		\end{algorithm}
	\end{minipage}
	\hfill
	\begin{minipage}[t]{0.49\textwidth}
		\begin{algorithm} [H]
			\caption{\sscn{}: Stochastic Subspace Cubic Newton \citep{hanzely2020stochastic}} \label{alg:sscn}
			\begin{algorithmic}
				\State \textbf{Requires:} Initial point $x^0 \in \Rd$, distribution of random matrices $\cD$, Lipschitzness of Hessian constant $\Ls >0$
				\For {$k=0,1,2\dots$}
				\vspace{2.1mm}
				\State {\color{mydarkgreen}Sample $\sk \sim \cD$}
				\vspace{2.1mm}
				\State $x^{k+1} = x^k - {\color{mydarkgreen}\sk} {\color{brown}\argmin_{h \in \Rd} \hat T_{\color{mydarkgreen}\sk}(x^k, h)}$\footnote{\vspace{1mm}\newline for {\color{brown}$\hat T_{\color{mydarkgreen}\s}(x, h) = \la \g(x), {\color{mydarkgreen}\s} h \ra + \frac 12 \normsM {{\color{mydarkgreen}\s} h} x + \frac {\Ls}6 \normM {{\color{mydarkgreen}\s} h}{} ^3$}.}
				\EndFor
			\end{algorithmic}
		\end{algorithm}
	\end{minipage}
	\caption[Pseudocodes of algorithms related to \sgn{}]{Pseudocodes of algorithms related to \sgn{}. We highlight the stepsizes of the Newton method in blue, subspace sketching in green, and regularized Newton step in brown.}
\end{figure}

\section{Is $\okd$ convergence rate possible for \sscn{}?} \label{sec:sscn_okd}

    \citet{hanzely2020stochastic} proposed \sscn{}, the \sap{} version of the cubic Newton method. While it intuitively seems that directly combining these approaches could directly lead to the desired global rate of $\mathcal{O}(k^{-2})$, achieving such a rate demands an extremely careful choice of assumptions and distribution of sketch matrices. Unfortunately, the authors of SSCN encountered a slight mismatch, resulting in a slower rate of $\mathcal{O}(k^{-1})$.

    We can present a slight modification of the \sscn{} algorithm to showcase what modification is needed to achieve the desired global rate of $\mathcal{O}(k^{-2})$.

    For functions $f$ that satisfy
    \begin{equation} \label{eq:sscn_modified_fclass}       
    \left | f(x+ \mathbf S h)-f(x) - \langle \nabla f(x), \mathbf S h\rangle  - \frac 12 \Vert \mathbf S h \Vert ^2 _ 2 \right |
    \leq \frac {L'}6 \Vert \mathbf S h \Vert ^3 _{2}, \qquad \forall \mathbf S \sim \mathcal D, \forall h \in \R^{\tau(\s)}, 
    \end{equation}
    the sequence of iterates is defined as
    \begin{equation} \label{eq:sscn_modified_update}
    x^{k+1} = \argmin_{ h\in \R^d} \left \{ f(x^k) + \langle \nabla f(x), \mathbf S h \rangle + \frac 12 \Vert \mathbf S h \Vert ^2 _2 + \frac {L'_{est}}6 \Vert \mathbf S h \Vert ^3 _{2}  \right\},
    \end{equation}
    for constant $L'_{est} \geq L'$ and sketch $\mathbf S$ sampled from distribution $\mathcal D$, s.t.
    \begin{equation}
    \mathbb E_{\mathbf S \sim \mathcal D} \left[\mathbf P_2 \stackrel{\text{def}}= \mathbf S (\mathbf S ^\top \mathbf S )^{\dagger} \mathbf S^\top \right ] = \frac \tau d \mathbf I,
    \qquad \text{which implies} \quad \mathbb E [\Vert \mathbf P_2 h \Vert^2_2] = \frac \tau d\Vert \mathbf h \Vert^2_2,
    \end{equation}
    achieves global convex convergence rate $\mathcal O(k^{-2}).$ However, while the left-hand side of \eqref{eq:semistrong_approx} is the second-order Taylor expansion, the left-hand side of \eqref{eq:sscn_modified_fclass} is not and we currently do not know which class of functions $f$ satisfies this requirement.

    In the case of the original SSCN, there is a discrepancy between usage of $l_2$ norms and local norms in the left-hand side of \eqref{eq:sscn_modified_fclass} and update rule \eqref{eq:sscn_modified_update}. This causes an extra quadratic term to appear in \eqref{eq:global_step}, resulting in the slower $\mathcal O(k^{-1})$ rate.

\section{Construction of the sketch distribution} \label{ssec:sketch_distribution}
Here we demonstrate that the distribution of sketching matrices satisfying \Cref{as:projection_direction} can be obtained from sketches with $l_2$--unbiased projection (which were used in \citep{hanzely2020stochastic}).

\begin{lemma} [Construction of sketch matrix $\s$] 
\label{le:setting_s}
    If we have a sketch matrix distribution $\tilde {\mathcal D} $ so that a projection on $\Range{\mathbf M}, \mathbf M \sim \tilde{\mathcal D}$ is unbiased in $l_2$ norms,
    \begin{equation}
        \mathbb E_{\mathbf M \sim \tilde {\cD} } \left[ \mathbf M \left( \mathbf M ^ \top \mathbf M \right) ^ \dagger \mathbf M ^ \top \right] = \frac \tau d \mI,
    \end{equation}
    then distribution $\cD$ of $\s$ defined as $\s^\top \eqdef \mathbf M \left[ \h(x) \right] ^{-1/2}$ (for $\mathbf M \sim \tilde {\cD}$) satisfy 
    \begin{equation}
        \mathbb E_{\s \sim \cD} \left[ \p \right] = \frac \tau d \mI.
    \end{equation}
\end{lemma}

\section{Proofs}

\subsection{Basic facts}
For any vectors $a, b\in\R^d$ and scalar $\nu>0$, Young's inequality states that
\begin{align}
	2\<a, b>\le \nu\|a\|^2+ \frac{1}{\nu}\|b\|^2. \label{eq:young}
\end{align}
Moreover, we have
\begin{align}
	\|a+b\|^2
	\le 2\|a\|^2 + 2\|b\|^2.\label{eq:a_plus_b}
\end{align}
More generally, for a set of $m$ vectors $a_1,\dotsc, a_m$ with arbitrary $m$, it holds
\begin{align}
	\Bigl\|\frac{1}{m}\sum_{i=1}^m a_i\Bigr\|^2
	\le \frac{1}{m}\sum_{i=1}^m\|a_i\|^2. \label{eq:jensen_for_sq_norms}
\end{align}
For any random vector $X$ we have
\begin{align}
	\ec{\|X\|^2}=\|\ec{X}\|^2 + \ec{\|X-\ec{X}\|^2}. \label{eq:rand_vec_sq_norm}
\end{align}
If $f$ is $L_f$-smooth, then for any $x,y\in\R^d$, it is satisfied
\begin{equation}
	f(y)
	\le f(x) + \<\nabla f(x), y-x> + \frac{L_f}{2}\|y-x\|^2. \label{eq:smooth_func_vals}
\end{equation}

Finally, for $L_f$-smooth and convex function $f:\R^d\to\R$, it holds
\begin{equation}
	f(x) \leq f(y)+\langle\nabla f(x), x-y\rangle-\frac{1}{2 L_f}\|\nabla f(x)-\nabla f(y)\|^{2} \label{eq:smooth_conv}.
\end{equation}

\begin{proposition} \label{pr:three_point} [Three-point identity]
	For any $u,v,w \in \R^d$, any $f$ with its Bregman divergence $D_f(x,y) = f(x)-f(y) - \langle \nabla f(y), x-y\rangle $, it holds
	\begin{align*}
		\langle \nabla f(u) - \nabla f(v), w-v \rangle
		&= D_f (v,u) + D_f(w,v) - D_f(w,u).
	\end{align*}
\end{proposition}

\begin{lemma}[Arithmetic mean -- Geometric mean inequality]
	For $c\geq0$ we have 
	\begin{equation} 
		1+ c = \frac {1 + (1+2c)}{2} \stackrel{AG} \geq \sqrt{1+2c}.
	\end{equation}
\end{lemma}

\subsection{Proof of \Cref{th:three}}
\begin{proof}
	Because $\g(x^k) \in \Range{\h(x^k)}$, it holds $\h(x^k)  [\h(x^k)]^\dagger \g(x^k) = \g(x^k)$.
	Updates \eqref{eq:sgn_aicn} and \eqref{eq:sgn_sap} are equivalent as
	\begin{eqnarray*}
		\pk k [\h(x^k)]^\dagger \g(x^k)
		&=& \sk \left( \sk^\top \h (x^k) \sk \right)^\dagger \sk^\top \h(x^k)  [\h(x^k)]^\dagger \g(x^k)\\
		&=& \sk \left( \sk^\top \h (x^k) \sk \right)^\dagger \sk^\top \g(x^k)\\
		&=&\sk [\hSk(x^k)]^\dagger \gSk(x^k).
	\end{eqnarray*}
	Taking gradient of $\modelSk{x^k, h}$ w.r.t. $h$ and setting it to $0$ yields that for solution $h^*$ holds
	\begin{gather} \label{eq:equiv_grad}
		\gSk(x^k) + \hSk(x^k) h^* + \frac \Lalg 2 \normM {h^*} {x^k, \sk} \hSk(x^k)h^* =0,
	\end{gather}
	which after rearranging is
	\begin{gather} \label{eq:equiv_hsol}
		h^* = - \left( 1+ \frac \Lalg 2 \normM {h^*} {x^k, \sk} \right)^ {-1} \left[\hSk(x^k) \right]^\dagger \gSk(x^k),
	\end{gather}
	thus solution of cubical regularization in local norms \eqref{eq:sgn_Ts} has form of Newton method with stepsize $\als k = \left( 1+ \frac \Lalg 2 \normM {h^*} {x^k, \sk} \right)^ {-1}$. We are left to show that this $\als k$ is equivalent to \eqref{eq:sgn_alpha}.
	
	Substitute $h^*$ from \eqref{eq:equiv_hsol} to \eqref{eq:equiv_grad} and $\als k = \left( 1+ \frac \Lalg 2 \normM {h^*} {x^k, \sk} \right)^ {-1}$ and then use $\h(x^k) [\h(x^k)]^\dagger \g(x^k) = \g(x^k)$, to get
	\begin{eqnarray*}
		0
		&=&\gSk(x^k) + \hSk(x^k) \left( - \als k \left[\hSk(x^k) \right]^\dagger \gSk(x^k) \right) \\
		&& \quad + \frac \Lalg 2 \left( \als k \normMSdk {\gSk(x^k)} {x^k}\right) \hSk(x^k) \left(- \als k \left[\hSk(x^k) \right]^\dagger \gSk(x^k) \right)\\
		&=& \left( 1 - \als k - \frac \Lalg 2 \als k^2 \normMSdk {\gSk(x^k)} {x^k} \right) \gSk(x^k).
	\end{eqnarray*}
	To conclude the proof, observe that stepsize $\als k$ from \eqref{eq:sgn_alpha} is set to be the positive root of polynomial $1 - \als k - \frac \Lalg 2 \als k^2 \normMSdk {\gSk(x^k)} {x^k}=0$. Because $\als k$ corresponds to $h^*$ such that $\left . \nabla_h \modelSk{x^k, h} \right|_{h*} =0$, vector $h^*$ is the minimizer of the regularized model $\modelSk{x^k, h}$ in \eqref{eq:sgn_reg}. On the other hand, the equation \eqref{eq:equiv_hsol} shows that $h^*$ has the form of Newton method with the stepsize \eqref{eq:sgn_aicn}.
 
 This concludes the equivalence of \eqref{eq:sgn_reg}, \eqref{eq:sgn_aicn}, and \eqref{eq:sgn_sap}.
\end{proof}

\subsection{Proof of \Cref{le:sketch_equiv}}
\begin{proof}
	For arbitrary square matrix $\mathbf M$ pseudoinverse guarantee $\mathbf M^ \dagger \mathbf M \mathbf M^ \dagger = \mathbf M^\dagger$. Applying this to $\mathbf M \leftarrow \left( \st \h (x) \s \right)$ yields $\la \p y, \p z \ra _ {\h(x)} = \la \p y, z \ra _{\h(x)} y,z \in \Rd $. Thus, $\p$ is really projection matrix w.r.t. $\normM \cdot x$. \\  
\end{proof}

\subsection{Proof of \Cref{le:tau_exp}}
\begin{proof}
	We follow proof of \citep[Lemma 5.2]{hanzely2020stochastic}.
	Using definitions and the cyclic property of the matrix trace,
	\begin{eqnarray*}        
		\E{\tau(\s)}
		&=& \E{\Tr\left(\mI^{\tau(\s)}\right)}
		= \E{\Tr\left( \s^\top \h(x) \s \left( \s^\top \h(x) \s \right)^{\dagger} \right)}\\
		&=& \E{\Tr\left( \p \right)}
		= \Tr \left(\frac \tau d \mI^d\right)
		= \tau.
	\end{eqnarray*}
\end{proof}

\subsection{Proof of \Cref{le:projection_contract}}
\begin{proof} 
    Equalities in \eqref{eq:proj_h} and \eqref{eq:proj_g} follows from directly expanding the definitions of norms $\normsM \cdot x, \normsMd \cdot x$ and $\p$ and using property of pseudoinverse $\mathbf M = \mathbf M \mathbf M^\dagger \mathbf M$ and $\mathbf M^\dagger = \mathbf M^\dagger \mathbf M \mathbf M^\dagger$ (for $\mathbf M = \st \h (x) \s$) and that $h,g,\h(x)$ are deterministic.
    \begin{eqnarray*}
        \E{\normsM{\p  h} x }
        &=& \E{h^\top {{\color{blue}\p^\top}} \h(x) {{\color{blue}\p}}  h}\\
        &=& \E{h^\top {{\color{blue}\h(x) \s \left( \st \h (x) \s \right)^\dagger \st}} \h(x) {{\color{blue} \s \left( \st \h (x) \s \right)^\dagger \st \h(x)}}  h}\\
        &=& \E{h^\top \h(x) \s \left( \st \h (x) \s \right)^\dagger \st \h(x)  h}\\
        &=& \E{h^\top \h(x) \p  h}\\
        &=&  h ^\top \h(x) \E \p h\\
        &\stackrel{As.\ref{as:projection_direction}} =& \frac \tau d \normsM h x,
    \end{eqnarray*}
            
    \begin{align*}
        \E{\normsMd{\p ^\top g} x }
        &= \E{g^\top {{\color{blue}\p}} [\h(x)]^\dagger {{\color{blue}\p^\top}}  g}\\
        &= \E{g^\top {{\color{blue} \s \left( \st \h (x) \s \right)^\dagger \st \h(x)}} [\h(x)]^\dagger {{\color{blue} \h(x) \s \left( \st \h (x) \s \right)^\dagger \st}}  g}\\
        &= \E{g^\top \s \left( \st \h (x) \s \right)^\dagger \st \h(x) \s \left( \st \h (x) \s \right)^\dagger \st  g}\\
        &= \E{g^\top \s \left( \st \h (x) \s \right)^\dagger \st  g}\\
        &= \E{g^\top \s \left( \st \h (x) \s \right)^\dagger \st {{\color{blue} \h(x) [\h(x)]^\dagger}}  g} \qquad \text{if } g \in \Range{\h(x)}\\        
        &= \E{g^\top \p [\h(x)]^\dagger  g}\\
        &=  g ^\top \E \p \h(x) g\\
        &=  g ^\top \E \p [\h(x)]^\dagger g\\
        &\stackrel{As.\ref{as:projection_direction}} = \frac \tau d \normsMd g x.\\
    \end{align*}
\end{proof}

\subsection{Proof of \Cref{le:setting_s}}

\begin{proof}
	We have
	\begin{eqnarray*}
		\EE_{\s \sim \cD}\left[\p\right]
		&=& \left[\h(x)\right]^{-1/2} \EE_{\mathbf M \sim \tilde {\cD}} \left[ \mathbf M^\top \left( \mathbf M^\top \mathbf M \right)^\dagger \mathbf M \right] \left[\h(x)\right]^{1/2}\\
		&=& \left[\h(x)\right]^{-1/2} \frac \tau d \mI \left[\h(x)\right]^{1/2}
		= \frac \tau d \mI.
	\end{eqnarray*}
\end{proof}

\subsection{Proof of \Cref{le:one_step_dec}}
\begin{proof}
	For $h^k = x^{k+1}-x^k$, we can follow proof of \citep[Lemma 10]{hanzely2022damped},
	\begin{eqnarray*}
		  f(x^k)-f(x^{k+1})
		&\stackrel{\eqref{eq:semistrong_approx}} \geq& - \la \gS(x^k),h^k \ra  - \frac 12 \normsM {h^k} {x^k, \sk} - \frac \Lalg 6 \normM {h} {x^k, \sk}^3 \\
		&\stackrel{\eqref{eq:transition-primdual}} =& \als k \normsMSdk {\gSk(x^k)} {x^k} - \frac 12 \als k^2 \normsMSdk {\gSk(x^k)} {x^k} \\
		&& \qquad - \frac \Lalg 6 \als k^3 \normM {\gSk(x^k)} {x^k, \sk}^{*3} \\
		& =& \left(1 - \frac 12 \als k - \frac \Lalg 6 \als k^2 \normMSdk {\gSk(x^k)} {x^k} \right) \als k \normsMSdk {\gSk(x^k)} {x^k}\\
		& \geq& \frac 12 \als k \normsMSdk {\gSk(x^k)} {x^k}\\
		& \geq& \frac 1{2 \max \left \{\sqrt{\Lalg \normMSdk {\gSk(x^k)} {x^k} }, 2 \right\} } \normsMSdk {\gSk(x^k)} {x^k}.
	\end{eqnarray*}
\end{proof}

\subsection{Proof of \Cref{le:local_step_subspace}}
\begin{proof}
	We bound norm of $\gS(x^{k+1})$ using basic norm manipulation and triangle inequality as
	\begin{eqnarray*}
		 \normMSdk {\gSk(x^{k+1})} {x^k}
		&=& \normMSdk{\gSk(x^{k+1}) - \hSk (x^k)(x^{k+1}-x^k)  - \als k \gSk (x^k) }{x^k}\\
		&=& \left \Vert \gSk(x^{k+1}) - \gSk (x^{k}) - \hSk (x^k)(x^{k+1}-x^k) +\right.\\
        && \hspace{5cm} \left. + (1-\als k) \gSk (x^k) \right \Vert _{x^k, \sk}^*\\
		&\leq& \normMSdk{\gSk(x^{k+1}) - \gSk(x^{k}) - \hSk(x^k)(x^{k+1}-x^k)}{x^k} \\
		&& \qquad  + (1-\als k) \normMSdk{\g(x^k) }{x^k}.
	\end{eqnarray*}
	
	Using $\Lsemi$--semi-strong self-concordance, we can continue
	\begin{eqnarray*}
		\cdots &\leq& \normMSdk{\gSk(x^{k+1}) - \gSk(x^{k}) - \hSk(x^k)(x^{k+1}-x^k)}{x^k} \\
		&& \qquad  + (1-\als k) \normMSdk{\gS(x^k) }{x^k}  \\
		&\leq& \frac{\Lsemi}{2} \normsM{x^{k+1}-x^k} {x^k, \sk} + (1-\als k) \normMSdk{\gSk(x^k) }{x^k} \\
		&=& \frac{\Lsemi \als k^2}{2} \normsMSdk{\gSk(x^k)} {x^k} + (1-\als k) \normMSdk{\gSk(x^k) }{x^k} \\
		&=&  \left(\frac{\Lalg \als k^2}{2} \normMSd{\gSk(x^k)} {x^k} -\als k +1\right) \normMSdk{\gSk(x^k) }{x^k}\\
		&\stackrel{\eqref{eq:sgn_alpha}}{=}& \Lalg \als k^2\normsMSdk{\gSk(x^k) }{x^k}.
	\end{eqnarray*}
	The last equality holds because of the choice of $\als k$.
\end{proof}

\subsection{Technical lemmas}

\begin{lemma}[Arithmetic mean -- Geometric mean inequality]
	For $c\geq0$ we have 
	\begin{equation} \label{eq:AG}
		1+ c = \frac {1 + (1+2c)}{2} \stackrel{AG} \geq \sqrt{1+2c}.
	\end{equation}
\end{lemma}

\begin{lemma}[Jensen for square root]
	Function $f(x) = \sqrt x$ is concave, hence for $c\geq0$ we have 
	\begin{equation}\label{eq:jensen}
		\frac 1 {\sqrt 2} (\sqrt c+1) \leq \sqrt{c+1} \qquad \leq \sqrt c + 1.
	\end{equation}
\end{lemma}

\subsection{Proof of \Cref{le:global_step}}
\begin{proof}
	Denote 
	\begin{equation*}
		\Omega_\s (x, h') \eqdef f(x) + \la \g(x), \p h' \ra + \frac 12 \normsM {\p h'} x + \frac {\Lalg}6 \normM {\p h'} x ^3,
	\end{equation*}
	so that 
	\begin{equation*}
		\min_{h' \in \Rd} \Omega_\s (x,h') = \min_{h\in \mathbb{R}^{\tau(\s)}} \modelS{x,h}.
	\end{equation*}
	
	For arbitrary $y \in \Rd$ denote $h \eqdef y-x^k$. We can calculate
	\[f(x^{k+1})  \leq  \min_{h'\in \mathbb{R}^{\tau(\s)}} \modelS {x^k, h'} = \min_{h'' \in \Rd} \Omega_\s (x^k,h''),\]
	and
	\begin{eqnarray*}
		\efa \E{f(x^{k+1})}\\
		&\leq& \E{ \Omega_\s (x^k, h)} \\
		& =& f(x^k) + \frac \tau d \la \g(x^k), h \ra + \frac 12 \E{\normsM{\pk k h} {x^k} }  + \E{\frac \Lalg 6 \normM{\pk k h} {x^k} ^3}\\
		& \stackrel{\eqref{eq:proj_h}} \leq& f(x^k) + \frac \tau d \la \g(x^k), h \ra + \frac \tau {2d} \normsM h {x^k}  + \frac {\Lalg} 6 \frac \tau d \normM h {x^k} ^3\\
		& \stackrel{\eqref{eq:semistrong_approx}}\leq& f(x^k) + \frac \tau d \left( f(y)-f(x^k) + \frac {\Lsemi}6 \normM{y-x^k}{x^k}^3 \right)  + \frac {\Lalg} 6 \frac \tau d \normM h {x^k} ^3.
	\end{eqnarray*}
	
	In second to last inequality depends on the unbiasedness of projection $\p$, Assumption~\ref{as:projection_direction}. The last inequality follows from semi-strong self-concordance, \Cref{pr:semistrong_stoch} with $\s=\mI$.
\end{proof}

\subsection{Proof of \Cref{th:global_convergence}}
\begin{proof}
	Denote
	\begin{eqnarray}
		A_0 & \eqdef& \frac 43 \left(\frac d \tau \right)^3,\\
		A_k &\eqdef& A_0 + \sumin tk t^2 = A_0 -1 + \frac{k(k+1)(2k+1)}{6} \geq A_0 +\frac {k^3} 3, \\
		&& \qquad \text{...consequently } \quad  \sumin tk \frac {t^6}{A_t^2} \leq 9k, \\
		\eta_t & \eqdef& \frac d \tau \frac{(t+1)^2} {A_{t+1}} \qquad \text{implying } 1-\frac \tau d \eta_t = \frac {A_t} {A_{t+1}}.
	\end{eqnarray}
	Note that this choice of $A_0$ implies (as in \citep{hanzely2020stochastic})
	\begin{align}
		\eta_{t-1} 
		\leq \frac d \tau \frac {t^2}{A_0 + \frac {t^3} 3 }
		\leq \frac d \tau \sup_{t\in \N} \frac {t^2}{A_0 + \frac {t^3} 3 }
		\leq \frac d \tau \sup_{\zeta>0} \frac {\zeta^2} {A_0 + \frac {\zeta^3} 3}
		= 1
	\end{align}
	and $\eta_t \in [0,1]$.    
	Set $y \eqdef \eta_t \opt + (1-\eta_t) x^t$ in \Cref{le:global_step}. From convexity of $f$,
	\begin{eqnarray*}        
		\E{f(x^{t+1} | x^t)} 
		& \leq&  \left(1- \frac \tau d \right) f(x^t) + \frac \tau d \fopt \eta_t + \frac \tau d f(x^t) (1-\eta_t) \\
		& &\qquad + \frac \tau d \left( \frac {\max \Ls + \Lsemi}  6 \normM{x^t-\opt} {x^t} ^3 \eta_t^3 \right).
	\end{eqnarray*}    
	Denote $\delta_t \eqdef \E{f(x^t) -\fopt}$. Subtracting $\fopt$ from both sides and substituting $\eta_k$ yields
	\begin{align}        
		\delta_{t+1}
		& \leq  \frac {A_t} {A_{t+1}} \delta_t + \frac {\max \Ls + \Lsemi}  6 \normM{x^t-\opt} {x^t} ^3 \left( \frac d \tau \right)^2 \left(\frac {(t+1)^2}{A_{t+1}} \right)^3 .
	\end{align}
	Multiplying by $A_{t+1}$ and summing from from $t=0,\dots, k-1  $ yields
	\begin{align}        
		A_{k} \delta_{k}
		& \leq  A_0 \delta_0 + \frac {\max \Ls + \Lsemi} 6 \frac {d^2} {\tau^2} \sum_{t=0}^{k-1} \normM{x^t-\opt} {x^t} ^3 \frac {(t+1)^6}{A_{t+1}^2}.
	\end{align}    
	Using $\sup_{x\in \level} \normM{x-\opt} {x} \leq R$ we can simplify and shift summation indices,
	\begin{eqnarray}        
		A_{k} \delta_{k}
		& \leq & A_0 \delta_0 + \frac {\max \Ls + \Lsemi} 6 \frac {d^2} {\tau^2} D^3 \sumin tk \frac {t^6}{A_t^2}\\
		& \leq&  A_0 \delta_0 + \frac {\max \Ls + \Lsemi} 6 \frac {d^2} {\tau^2} D^3 9k,
	\end{eqnarray}
	and
	\begin{eqnarray}        
		\delta_{k}
		& \leq&  \frac {A_0 \delta_0}{A_k} + \frac {3(\max \Ls + \Lsemi) d^2 D^3 k} {2 \tau^2 A_k}\\
		& \leq & \frac {3A_0 \delta_0}{k^3} + \frac {9(\max \Ls + \Lsemi) d^2 D^3} {2 \tau^2 k^2},
	\end{eqnarray}
	which concludes the proof.
\end{proof}

\subsection{Proof of \Cref{th:local_linear}}
Before we start proof, we first state that for self-concordant functions (\Cref{def:self-concordance}) we can bound function value suboptimality by the norm of the gradient in the neighborhood of the solution.

\begin{proposition} \citep[Lemma D.3]{hanzely2020stochastic}
	\label{pr:sscn_lemmas}
	For any $\gamma>0$ and $x^k$ in neighborhood
	$x^k \in \left \{ x: \normMd {\g(x)} {x} < \frac 2{(1+\gamma^{-1})\Lstandard } \right \}$ for $\Lstandard$--self-concordant function $f:\R^d\to\R$, we can bound
	\begin{equation} \label{eq:sc_gamma}
		f(x^k)-\fopt \leq \frac 12 (1+ \gamma) \normsMd {\g(x^k)} {x^k}.
	\end{equation}
\end{proposition}

\begin{proof}[Proof of \Cref{th:local_linear}]
    Note 
     \begin{eqnarray}     
         \normsMSdk {\gSk(x^k)} {x^k}
         &=& \g(x^k)^\top \s \left( \st \h(x) \s \right) ^\dagger \st \g(x^k) \nonumber\\
         &=& \g(x^k)^\top \s \left( \st \h(x) \s \right) ^\dagger \st \h(x) \s \left( \st \h(x) \s \right) ^\dagger \st \g(x^k) \nonumber\\
         &=& \g(x^k)^\top \s \left( \st \h(x) \s \right) ^\dagger \st \h(x)[\h(x)]^\dagger \h(x) \cdot \nonumber\\
         && \hspace{5cm } \cdot\s \left( \st \h(x) \s \right) ^\dagger \st \g(x^k) \nonumber\\
         &=& \normsMd{\pk k ^\top \g(x^k)} {x^k}  \label{eq:gs_to_proj}\\
         &\stackrel{\eqref{eq:proj_contr}}\leq& \normsMd{\g(x^k)} {x^k}, \nonumber
     \end{eqnarray}
     and for $\Lstandard$--self-concordant function $f$ and $\gamma>0$ in the neighborhood $x^k \in \left \{ x: \normMd {\g(x)} {x} < \frac 2{(1+\gamma^{-1})\Lstandard } \right \}$, we have
     \begin{equation}         
         \normMSdk {\gSk(x^k)} {x^k}
         \leq \normMd {\g(x^k)} {x^k}
         < \frac 2{(1+\gamma^{-1})\Lstandard}.
     \end{equation}

     From \Cref{eq:one_step_dec} we have 
    \begin{align}
        f(x^k)-f(x^{k+1}) 
        \geq \frac 1 {a_k} \normsMSdk {\gSk(x^k)} {x^k} 
        > \frac 1 {2b} \normsMSdk {\gSk(x^k)} {x^k} 
    \end{align}
    where 
    \begin{equation*}
    a_k 
    \eqdef 2 \max \left \{\sqrt{\Lalg \normMSdk {\gSk(x^k)} {x^k} } ,2 \right\}
    < 4 \max \left \{\sqrt{\frac \Lalg{2(1+\gamma^{-1})\Lstandard} } ,1 \right\} \quad \eqdef {2b}
    \end{equation*}
    
We can take the expectation and continue    
	\begin{eqnarray*}
		\E{f(x^k)-f(x^{k+1})} 
		& \stackrel{\eqref{eq:one_step_dec}} \geq& \E{\frac 1{2b} \normsMSd {\gSk(x^k)} {x^k}}\\
		&\stackrel{\eqref{eq:gs_to_proj}}=& \E{ \frac 1{2b} \normsMd {\pk k ^\top \g(x^k)} {x^k} } \\
		&\stackrel{\eqref{eq:proj_g}}=& \frac \tau {2bd} \normsMd {\g(x^k)} {x^k}\\
		&\stackrel{\eqref{eq:sc_gamma}}\geq& \frac {\tau} {bd(1+ \gamma)} (f(x^k)-\fopt).
	\end{eqnarray*}
	Hence 
	\[ \E{f(x^{k+1})-\fopt)} \leq \left( 1- \frac {\tau} {bd(1+\gamma)} \right) (f(x^k)-\fopt), \]
	and to finish the proof, we choose $\gamma=1$ and use tower property across iterates $x^0, x^1, \dots, x^k$.
 
 As semi-strong self-concordance is stronger than standard self-concordance (see \Cref{sec:concordances}) and $\Lsemi\geq \Lstandard$, for simplicity of presentation we replace $\Lsemi$ by $\Lstandard$.
\end{proof}

\subsection{Towards proof of \Cref{th:global_linear}} \label{ssec:towards_glob}

\begin{proposition}\citep[Equation (47)]{RSN}
	Relative convexity \eqref{eq:rel_conv} implies bound
	\begin{equation}
		\fopt \leq f(x^k) - \frac 1 {2\murel} \normsMd{\g(x^k)}{x^k}. \label{eq:rel_conv_dec}
	\end{equation}
\end{proposition}

\begin{proposition} Analogy to \citep[Lemma 7]{RSN} \label{pr:rho}
	For $\s \sim \cD $ satisfying conditions
	\begin{equation}
		\Null{\s^\top \h(x) \s} = \Null{\s} \quad \text{and} \quad \Range{\h(x)} \subset \Range {\E{\s_k \s_k ^\top} },
	\end{equation}
	also, the exactness condition holds
	\begin{equation}
		\Range{\h(x)} = \Range{\E{\hat \p}},
	\end{equation}
	and formula for $\rho(x)$ can be simplified
	\begin{equation}
		\rho(x) = \lambda_{\text{min}}^+ \left( \E{\alpha_{x, \s} \p} \right) >0
	\end{equation}
	and bounded $0 < \rho(x) \leq 1$. Consequently, $0 < \rho \leq 1$.
\end{proposition}

\begin{lemma}[Stepsize bound] \label{le:stepsize_bound}
	Stepsize $\als k$ can be bounded as 
	\begin{equation}
		\als k        
		\leq  \frac {\sqrt 2} {\sqrt{\Lalg \normMSdk{\gSk(x^k)} {x^k}}},
	\end{equation}
	and for $x^k$ far from solution, $\normMSdk{\gSk(x^k)} {x^k} \geq \frac 1 {L_{\sk}}$ and $\Lalg \geq \frac 92 \sup_\s {L_{\s}} \Lrel_{\s}^2$ holds
	$
	\als k \Lrels
	\leq \frac 23.
	$
\end{lemma}
\begin{proof}[Proof of \Cref{le:stepsize_bound}]
	Denote $G_k \eqdef \Lalg \normMSdk{\gSk(x^k)} {x^k}$. Using \eqref{eq:jensen} with $c \leftarrow 2 G > 0$ and
	\begin{equation}
		\als k
		= \frac {-1 + \sqrt{1+2 G}}{G}
		\leq \frac{\sqrt {2G}}{G}
		= \frac {\sqrt 2} {\sqrt G} 
		= \frac {\sqrt 2} {\sqrt{\Lalg \normMSdk{\gSk(x^k)} {x^k}}}
	\end{equation}
	and 
	\begin{align*}        
		\als k \Lrels
		&\leq \frac {\sqrt 2 \Lrels} {\sqrt{\Lalg \normMSdk{\gSk(x^k)} {x^k}}}\\
		&\leq \frac {\sqrt 2 \Lrels} {\sqrt{\frac 92 L_{\sk} \Lrels^2 \normMSdk{\gSk(x^k)} {x^k}}}\\
		& \leq \frac 23 \frac 1 {\sqrt{L_{\sk} \normMSdk{\gSk(x^k)} {x^k}}} 
		\leq \frac 23, \qquad \text{for } \normMSdk{\gSk(x^k)} {x^k} \geq \frac 1 {L_{\sk}}.
	\end{align*}
\end{proof}

\subsubsection{Proof of \Cref{th:global_linear}}
\begin{proof}
	Replacing $x\leftarrow x^k$ and $h \leftarrow \als k \pk k [\h(x^k)]^\dagger \g(x^k)$ so that $x^{k+1} = x^k + \s h$ in \eqref{eq:rels_smooth} yields
	\begin{eqnarray}
		f(x^{k+1})
		&\leq& f(x^k) - \als k \left(1 -  \frac 12 \Lrels \als k \right) \normsMSdk {\gSk(x^k)} {x^k}\\
		&\leq&  f(x^k) - \frac 23 \als k \normsMSdk {\gSk(x^k)} {x^k}.
	\end{eqnarray}
	In last step, we used that $\Lrels \als k \leq \frac 23$  holds for $\normMSdk {\gSk(x^k)} {x^k} \geq \frac 1 {\Lrels}$ (\Cref{le:stepsize_bound}). Next, we take an expectation over $x^k$ and use the definition of $\rho(x^k)$. 

    \begin{eqnarray}
		\E{f(x^{k+1})} 
		&\leq& f(x^k) - \frac 23 \normsM {\g(x^k)} {\E{\als k \sk \left[\hSk(x^k) \right]^\dagger \sk^\top}}\\
		&=& f(x^k) - \frac 23 \g(x^k)^\top {\E{\als k \sk \left[\hSk(x^k) \right]^\dagger \sk^\top}} \g(x^k)\\
		&=& f(x^k) - \frac 23 \g(x^k)^\top {\E{\als k \sk \left[\hSk(x^k) \right]^\dagger \sk^\top }\h(x^k) [\h(x^k)]^\dagger} \g(x^k)\\
		&=& f(x^k) - \frac 23 \g(x^k)^\top {\E{\als k \pk k} [\h(x^k)]^\dagger} \g(x^k)\\
		&\leq& f(x^k) - \frac 23 \rho(x^k) \g(x^k)^\top [\h(x^k)]^\dagger \g(x^k)\\
		&=& f(x^k) - \frac 23 \rho(x^k) \normsMd {\g(x^k)} {x^k}\\
		& \stackrel{\eqref{eq:rel_conv_dec}} \leq& f(x^k) - \frac 43 \rho(x^k) \murel \left( f(x^k) - \fopt \right).
	\end{eqnarray}
 
	Now $\rho(x^k) \geq \rho$, and $\rho$ is bounded in \Cref{pr:rho}. Rearranging and subtracting $\fopt$ gives
	\begin{equation}        
		\E{f(x^{k+1}) - \fopt}  \leq \left(1 - \frac 43 \rho \murel \right) (f(x^k) - \fopt),
	\end{equation}
	which after using tower property across all iterates yields the statement.
\end{proof}

\ifneurips
    \include{styles/neurips_2024_checklist}
\fi
\end{document}